\documentclass{siamart251216}

\usepackage{microtype}
\usepackage{amssymb}
\usepackage{bm}

\usepackage{epstopdf}
\ifpdf
  \DeclareGraphicsExtensions{.eps,.pdf,.png,.jpg}
\else
  \DeclareGraphicsExtensions{.eps}
\fi
\usepackage{xcolor}
\usepackage{tikz}
\usepackage{float}
\usepackage{graphicx}
\usepackage[caption = false]{subfig}
\graphicspath{{figure/}} 

\usepackage{algorithmic}
\newsiamremark{remark}{Remark}
\newsiamthm{coro}{Corollary}
\newsiamthm{prop}{Proposition}

\newcommand{\newsiamdefn}[2]{
  \theoremstyle{plain}
  \theoremheaderfont{\normalfont\sc}
  \theorembodyfont{\normalfont}
  \theoremseparator{.}
  \theoremsymbol{}
  \newtheorem{#1}{#2}
}
\newsiamdefn{defn}{Definition}
\newsiamdefn{amp}{Assumption}

\usepackage{amsfonts}
\usepackage{mathtools}
\usepackage{nicefrac}
\usepackage{mathrsfs}
\usepackage{nicematrix}
\usepackage{amsopn}
\usepackage{amsmath,booktabs,multirow}

\DeclareMathOperator*{\argmin}{arg\,min}

\def\R{\mathbb{R}}

\def\M{\mathcal{M}}

\allowdisplaybreaks[4]

\headers{FEM for GPEVP}{T. Chu and M. Zhang}

\title{Numerical analysis of linear finite element methods for the Gross--Pitaevskii eigenvalue problem
}

\author{
  Tianyang Chu\thanks{School of Mathematical Sciences, Chengdu University of Technology, Chengdu, Sichuan, 610059, China (\email{tchu@cdut.edu.cn}).}
  \and Mingyan Zhang\thanks{Corresponding author. School of Mathematics and Statistics, Northeast Normal University, Changchun 130024, China (\email{myzhang106@nenu.edu.cn}).}
}

\begin{document}

\maketitle

\begin{abstract}
In this paper, we analyze the standard and the mass-lumped linear finite element methods for the Gross--Pitaevskii eigenvalue problem. For the standard method, we give a new proof of the uniqueness of the discrete ground state up to sign for sufficiently small mesh sizes, which only requires the $H^1$-convergence of the discrete ground states. For both methods, we establish optimal convergence rates in the $L^p$- and $W^{1,p}$-norms for $1 \le p < \infty$. Under additional regularity assumptions on the potential, we further obtain an optimal convergence rate in the $W^{1,\infty}$-norm and an almost optimal convergence rate, up to a logarithmic factor, in the $L^\infty$-norm. Moreover, we prove that the discrete ground states superconverge to the Ritz projection of the exact ground state in the $W^{1,p}$-seminorm. The analysis is based on a perturbation argument. Numerical experiments confirm the theoretical results.
\end{abstract}



\section{Introduction}\label{sec: intro}

In this paper, we consider the Gross--Pitaevskii eigenvalue problem
\begin{align}\label{equ: GPEVP}
    - \Delta u + V u + \beta |u|^2u = \lambda u \quad \text{in } \Omega, \qquad u = 0 \quad \text{on } \partial\Omega, \qquad \|u\|_{L^2(\Omega)} = 1,
\end{align}
where $\Omega \subset \R^d$ with $d \in \{1,2,3\}$ is a bounded domain, $V$ is a trapping potential, and $\beta \in \R$ is the interaction parameter. Problem \eqref{equ: GPEVP} arises in quantum physics as a model for stationary states of Bose--Einstein condensates. Among these stationary states, the \emph{ground state} is of particular importance. It is defined as a minimizer of the Gross--Pitaevskii energy
\begin{align*}
    E(v) := \int_{\Omega} \Big( \frac12 |\nabla v|^2 + \frac12 V |v|^2 + \frac{\beta}{4} |v|^4 \Big) \,\mathrm{d}x
\end{align*}
under the normalization constraint, i.e., as a solution of
\begin{align}\label{equ: minimization}
    \min_{v \in H_0^1(\Omega)} E(v) \qquad \text{subject to} \quad \|v\|_{L^2(\Omega)} = 1.
\end{align}
Every minimizer of \eqref{equ: minimization} satisfies the Euler--Lagrange equation \eqref{equ: GPEVP}, where the eigenvalue $\lambda$ is the associated Lagrange multiplier.  Hence, the ground state is an eigenfunction of \eqref{equ: GPEVP}, and the corresponding eigenvalue is referred to as the ground state eigenvalue. Since the ground state is in general not available in closed form, its efficient and reliable numerical approximation is of central importance. 

The numerical computation of the ground state involves two main aspects: the spatial discretization of the problem and the design of iterative solvers for the resulting nonlinear discrete problem. Regarding the latter, one may start from either the eigenvalue problem \eqref{equ: GPEVP} or the constrained minimization problem \eqref{equ: minimization}. Starting from \eqref{equ: GPEVP}, the spatial discretization leads to a nonlinear algebraic eigenvalue problem. A widely used approach for such problems is the self-consistent field (SCF) iteration \cite{cances2000SCF,cances2000on,dion2007ground}, which only requires the solution of a linear eigenvalue problem in each iteration. Starting from \eqref{equ: minimization}, a variety of optimization-based approaches are available, including Riemannian optimization methods \cite{altmann2022energy,altmann2024riemanniannewton,danaila2017computation}, preconditioned nonlinear conjugate gradient methods \cite{antoine2017efficient}, regularized Newton methods \cite{wu2017regularized}, $J$-methods \cite{altmann2021j,jarlebring2014inverse}, and gradient flow methods \cite{bao2004computing,chen2024convergence,chu2026gradient,danaila2010new,henning2020sobolev,henning2023dependency,kazemi2010minimizing,zhang2022exponential}. For a comprehensive overview of numerical methods for computing the ground state, we refer to the recent review \cite{henning2025gross}.

In this paper, we focus on the first aspect, namely, the spatial discretization and its numerical analysis. Existing results in this direction include both a priori and a posteriori error estimates. General finite-dimensional approximations of the Gross--Pitaevskii eigenvalue problem were studied by Zhou in \cite{zhou2004analysis}, while related nonlinear eigenvalue problems were studied in \cite{chen2010numerical,zhou2007finite}.  These studies established optimal $H^1$-error estimates for the ground state approximations. These results were refined in \cite{cances2010numerical}, where sharp estimates for the ground state energy and eigenvalue were established, together with $L^2$- and $H^1$-error estimates for the ground state. Discretizations by the standard finite element method (FEM) were analyzed in \cite{chen2011finite,henning2025discrete}, while a mass-lumped linear FEM that preserves the positivity of the discrete ground state was proposed and analyzed in \cite{hauck2026positivity}. Generalized finite element approximations of the Gross--Pitaevskii problem were studied in \cite{henning2014two,henning2023optimal}, and finite-dimensional approximations of the Kohn--Sham equation were analyzed in \cite{chen2013numerical}. Besides these a priori results, a posteriori error estimates and adaptive FEMs were developed in \cite{chen2011adaptive,chen2011finite,dusson2017posteriori}, and adaptive FEMs for the Kohn--Sham equation were proposed in \cite{chen2014adaptive}. More recently, it was shown in \cite{hassan2025relation} that Galerkin approximations of the ground state superconverge to the $H^1$-orthogonal projection of the exact ground state.

In this work, we analyze linear finite element discretizations of \eqref{equ: GPEVP}, including both the standard FEM and the mass-lumped FEM. For the standard FEM, the uniqueness of the discrete ground state up to sign for sufficiently small mesh sizes can be deduced from the analysis in \cite{chen2011finite,henning2025discrete}. Here we present an alternative proof based on a perturbation argument for the linearized operator, which only requires the $H^1$-convergence of the discrete ground states. In \cite{chen2011finite,dai2015convergence,dai2008convergence,he2011convergence}, a perturbation approach was used to transfer a posteriori error estimates and adaptive finite element techniques from the associated source problem to the eigenvalue problem. Following a similar strategy, we combine this perturbation argument with existing $L^2$- and $H^1$-error estimates to derive optimal error estimates in the $L^p$- and $W^{1,p}$-norms for both methods, as well as superconvergence of the discrete ground states to the Ritz projection of the exact ground state in the $W^{1,p}$-seminorm.

The remainder of this paper is organized as follows. Section~\ref{sec: Preliminaries} collects some preliminaries, including properties of the Gross--Pitaevskii eigenvalue problem, the finite element discretization, and several technical results. In Section~\ref{sec: standard FEM}, we prove the uniqueness of the discrete ground state of the standard FEM up to sign for sufficiently small mesh sizes and establish convergence rates in the $L^p$- and $W^{1,p}$-norms as well as superconvergence results. Section~\ref{sec: mass-lumped FEM} is devoted to the corresponding error analysis and superconvergence results for the mass-lumped FEM. Numerical experiments illustrating the theoretical results are presented in Section~\ref{sec: numerical experiments}, and concluding remarks are given in Section~\ref{sec: conclusion}.

\section{Preliminaries}\label{sec: Preliminaries}

We adopt the standard notation for Sobolev spaces $W^{k,p}(\Omega)$, equipped with the norms $\|\cdot\|_{W^{k,p}(\Omega)}$ and semi-norms $|\cdot|_{W^{k,p}(\Omega)}$. For $k = 0$, we write $L^p(\Omega) := W^{0,p}(\Omega)$, and for $p = 2$, we write $H^k(\Omega) := W^{k,2}(\Omega)$. The inner product on $L^2(\Omega)$ is denoted by $(\cdot,\cdot)_{L^2(\Omega)}$. The space $H_0^1(\Omega)$ consists of functions in $H^1(\Omega)$ with zero trace on $\partial\Omega$. Throughout this paper, $C$ denotes a generic positive constant independent of the mesh size $h$, whose value may vary from one occurrence to another.

\subsection{The Gross--Pitaevskii eigenvalue problem}

Throughout this paper, the computational domain is taken to be $\Omega=(a,b)^d$. We assume that $V\in L^\infty(\Omega)$, $V\ge 0$ almost everywhere in $\Omega$, and $\beta\ge 0$. Denote by 
\begin{align*}
    \M:= \{ u \in H_0^1(\Omega): \|u\|_{L^2(\Omega)} = 1 \}.
\end{align*}
Under the above assumptions, the following result holds; see \cite{bao2013mathematical,cances2010numerical,henning2025gross}.
\begin{prop}
   The Gross--Pitaevskii eigenvalue problem \eqref{equ: GPEVP} admits infinitely many real eigenvalues $\{\lambda_k\}_{k\ge 1}$ with corresponding normalized eigenfunctions $\{u_k\}_{k\ge 1}\subset \M$. These eigenvalues can be ordered so that
\begin{align*}
    \lambda_1  < \lambda_2 \le \cdots \le \lambda_k \le \cdots, \qquad \lambda_k \to \infty \quad \text{as } k\to\infty.
\end{align*}
   In particular, the smallest eigenvalue $\lambda_1$ is simple, and the corresponding eigenfunction $u_1$ is the ground state, which is unique up to sign. Moreover, either $u_1>0$ or $u_1<0$ throughout $\Omega$.
\end{prop}

Let $u^*$ denote the $L^2$-normalized positive ground state of \eqref{equ: GPEVP} and let $\lambda^* := \lambda_1$ be the associated ground state eigenvalue. We introduce the following linear eigenvalue problem: find $\mu \in \mathbb{R}$ and $u \neq 0$ such that
\begin{align}\label{equ: LEVP}
    - \Delta u + Vu + \beta |u^*|^2 u = \mu u.
\end{align}
It follows directly from \eqref{equ: GPEVP} that $(\lambda^*,u^*)$ is an eigenpair of \eqref{equ: LEVP}. Moreover, the following result holds; see, e.g., \cite{cances2010numerical,henning2023dependency}.
\begin{prop}\label{prop: simple eigenvalue}
    The smallest eigenvalue $\mu_1$ of \eqref{equ: LEVP} is simple and coincides with the ground state eigenvalue of \eqref{equ: GPEVP}, i.e., $\mu_1 = \lambda^*$.
\end{prop}

Note that the ground state $u^*$ satisfies
\begin{align}\label{equ: Possion equation}
    - \Delta u^* = (\lambda^* - V - \beta|u^*|^2)u^* \quad \text{in } \Omega, \qquad u^* = 0 \quad \text{on } \partial\Omega,
\end{align}
which can be viewed as a Poisson equation for $u^*$. Since $V \in L^\infty(\Omega)$ and $H_0^1(\Omega) \hookrightarrow L^6(\Omega)$ for $d \le 3$, the right-hand side of \eqref{equ: Possion equation} belongs to $L^2(\Omega)$. Hence, the $H^2$-regularity result for convex domains \cite{grisvard2011elliptic} yields $u^* \in H^2(\Omega)$. 


\subsection{Finite element discretization}

Let $\mathcal{T}_h$ be a conforming, shape-regular, and quasi-uniform simplicial triangulation of $\Omega$ with mesh size $h:=\max_{K\in\mathcal{T}_h}\operatorname{diam}(K)$. We denote by $\mathcal{N}_h$ the set of all vertices of $\mathcal{T}_h$ and, for each element $K\in\mathcal{T}_h$, by $\mathcal{N}_K$ the set of vertices of $K$. Let
\begin{align*}
    V_h := \{ v_h \in C^0(\overline{\Omega}) : v_h|_K \in \mathbb{P}_1(K) \ \text{for all } K\in\mathcal{T}_h \}
\end{align*}
be the continuous piecewise linear Lagrange finite element space associated with $\mathcal{T}_h$, and define $V_h^0:=V_h\cap H_0^1(\Omega)$.

For $v \in C(\overline{\Omega})$, let $I_h v \in V_h$ denote its nodal interpolant. Since $H^2(\Omega) \hookrightarrow C(\overline{\Omega})$ for $d \le 3$, $I_h v$ is well defined for all $v \in H^2(\Omega)$, and the following estimate holds; see, e.g., \cite{brenner2008mathematical}.
\begin{prop}
    There exists a constant $C>0$, independent of $h$, such that for all $v \in H^2(\Omega)$,
    \begin{align}
        \|v - I_hv\|_{L^2(\Omega)} + h\|v - I_h v\|_{H^1(\Omega)} \le Ch^2 \|v\|_{H^2(\Omega)}.
    \end{align}
\end{prop}

For $u \in W_0^{1,1}(\Omega)$, let $R_h u \in V_h^0$ denote its Ritz projection, defined by
\begin{align}
    \int_\Omega \nabla R_hu \cdot \nabla v_h \,\mathrm{d}x = \int_\Omega \nabla u \cdot \nabla v_h \,\mathrm{d}x \qquad \forall\, v_h \in V_h^0;
\end{align}
see, e.g., \cite{diening2024pointwise}. This is well defined since $\nabla v_h \in L^\infty(\Omega)$ for all $v_h \in V_h^0$. Moreover, since $\Omega$ is convex, the following estimate holds; see, e.g., \cite{brenner2008mathematical}.
\begin{prop}
    There exists a constant $C>0$, independent of $h$, such that for all $u \in H^2(\Omega) \cap H_0^1(\Omega)$,
    \begin{align}
        \|u - R_h u\|_{L^2(\Omega)} + h\|\nabla(u - R_hu)\|_{L^2(\Omega)} \le C h^2 \|u\|_{H^2(\Omega)}.
    \end{align}
\end{prop}

To establish convergence in $L^p$ and $W^{1,p}$, we need the following stability and error estimates for the Ritz projection. Here and below, we assume that $h$ is sufficiently small.
\begin{prop}\label{prop: stability of the Ritz projection}
    Let $1 < p \le \infty$. There exists a constant $C_p>0$, independent of $h$, such that for all $u \in W_0^{1,p}(\Omega)$,
    \begin{align}
        \|\nabla R_h u\|_{L^p(\Omega)} \le C_p \|\nabla u\|_{L^p(\Omega)}.
    \end{align}
    Moreover, for all $u \in L^{\infty}(\Omega) \cap H_0^1(\Omega)$,
    \begin{align}
        \|R_hu\|_{L^{\infty}(\Omega)} \le C |\ln h| \|u\|_{L^{\infty}(\Omega)}.
    \end{align}
\end{prop}
The first estimate is proved in \cite[Corollary~3.2]{diening2024pointwise} for $d=2,3$ (see also \cite{rannacher1982some} for $2 \le p \le \infty$ and $d=2$). The second one is proved for smooth $u$ in \cite[Theorem~2]{schatz1980weak} for $d=2$ and in \cite[Theorem~12]{leykekhman2016finite} for $d=3$. By a standard density argument, it extends to all $u \in L^\infty(\Omega) \cap H_0^1(\Omega)$.

\begin{prop}\label{prop: p convergence of Ritz}
    Let $1 \le p < \infty$ and set $q := \max\{p,2\}$. There exists a constant $C_p>0$, independent of $h$, such that for all $u \in W^{2,q}(\Omega) \cap H_0^1(\Omega)$,
    \begin{align*}
        \|u - R_h u\|_{L^p(\Omega)} + h\|\nabla (u - R_hu)\|_{L^p(\Omega)} \le C_p h^2 \|u\|_{W^{2,q}(\Omega)}.
    \end{align*}
    Moreover, there exists a constant $C>0$, independent of $h$, such that for all $u \in W^{2,\infty}(\Omega) \cap H_0^1(\Omega)$,
    \begin{align*}
        \|u - R_h u\|_{L^{\infty}(\Omega)} \le C h^2 |\ln h|\|u\|_{W^{2,\infty}(\Omega)}, \qquad \|\nabla (u - R_hu)\|_{L^{\infty}(\Omega)} \le C h \|u\|_{W^{2,\infty}(\Omega)}.
    \end{align*}
\end{prop}

For $1 \le p < 2$, the first estimate follows from the case $p=2$ and H\"older's inequality, since $\Omega$ is bounded. For $2 \le p < \infty$, it is proved in \cite{rannacher1982some} for $d=2$, and for $d=3$ it follows from Proposition~\ref{prop: stability of the Ritz projection}, standard interpolation estimates, and a duality argument based on the $W^{2,p'}$-regularity of the Laplacian on convex domains \cite{Dauge1988,grisvard2011elliptic}. The $L^\infty$-estimates follow from Proposition~\ref{prop: stability of the Ritz projection}, the identity $u - R_hu = (I - R_h)(u - I_hu)$, and standard interpolation estimates. For $d = 1$, the estimates in Propositions~\ref{prop: stability of the Ritz projection} and~\ref{prop: p convergence of Ritz} are elementary since $R_h u = I_h u$.

\subsection{Some technical results}

In this subsection, we collect some technical results that will be used in the subsequent analysis. The first one is a $W^{1,p}$-regularity result for the Poisson problem with a right-hand side in divergence form on convex domains; see \cite[Corollary~1]{fromm1993potential}.
\begin{lemma}\label{lem: regularity}
    Let $2 \le p < \infty$ and $\bm{F} \in [L^p(\Omega)]^d$, and let $u \in H_0^1(\Omega)$ satisfy
    \begin{align*}
        (\nabla u, \nabla v)_{L^2(\Omega)} = (\bm{F}, \nabla v)_{L^2(\Omega)} \qquad \forall\, v \in H_0^1(\Omega).
    \end{align*}
    Then $u \in W_0^{1,p}(\Omega)$, and there exists a constant $C_p>0$, independent of $u$ and $\bm{F}$, such that
    \begin{align*}
        \|\nabla u\|_{L^p(\Omega)} \le C_p \|\bm{F}\|_{L^p(\Omega)}.
    \end{align*}
\end{lemma}

Combining Lemma~\ref{lem: regularity} with the stability of the Ritz projection, we obtain the following discrete regularity result.
\begin{lemma}\label{lem: general framework}
    Let $g \in L^2(\Omega)$ and $\bm{F} \in [L^6(\Omega)]^d$, and let $z_h \in V_h^0$ satisfy
    \begin{align*}
        (\nabla z_h, \nabla v_h)_{L^2(\Omega)} = (g, v_h)_{L^2(\Omega)} + (\bm{F}, \nabla v_h)_{L^2(\Omega)} \qquad \forall\, v_h \in V_h^0.
    \end{align*}
    Then there exists a constant $C>0$, independent of $h$, $g$, and $\bm{F}$, such that
    \begin{align*}
        \|\nabla z_h\|_{L^{6}(\Omega)}  \le C \big(\|g\|_{L^2(\Omega)} + \|\bm{F}\|_{L^6(\Omega)}\big).
    \end{align*}
\end{lemma}
\begin{proof}
     Let $w_g, w_{\bm{F}} \in H_0^1(\Omega)$ be the solutions of
    \begin{align*}
        (\nabla w_g, \nabla v)_{L^2(\Omega)} = (g, v)_{L^2(\Omega)}, \quad (\nabla w_{\bm{F}}, \nabla v)_{L^2(\Omega)} = (\bm{F}, \nabla v)_{L^2(\Omega)}, \quad \forall\, v \in H_0^1(\Omega).
    \end{align*}
    By the $H^2$-regularity on convex domains, Lemma \ref{lem: regularity} with $p  = 6$, and the Sobolev embedding $H^1(\Omega) \hookrightarrow L^6(\Omega)$ for $d \le 3$, we have
    \begin{align*}
        \|\nabla w_g\|_{L^6(\Omega)} \le C \|w_g\|_{H^2(\Omega)} \le C \|g\|_{L^2(\Omega)}, \qquad \|\nabla w_{\bm F}\|_{L^6(\Omega)} \le C \|\bm{F}\|_{L^6(\Omega)}.
    \end{align*}
    By the definition of $R_h$, for all $v_h \in V_h^0$,
    \begin{align*}
        (\nabla R_h(w_g + w_{\bm{F}}), \nabla v_h)_{L^2(\Omega)} = (g, v_h)_{L^2(\Omega)} + (\bm{F}, \nabla v_h)_{L^2(\Omega)},
    \end{align*}
    and hence $z_h = R_h(w_g + w_{\bm{F}})$. By Proposition~\ref{prop: stability of the Ritz projection} with $p = 6$, we obtain
    \begin{align*}
        \|\nabla z_h\|_{L^6(\Omega)} \le C \|\nabla (w_g + w_{\bm{F}}) \|_{L^6(\Omega)} \le C (\|g\|_{L^2(\Omega)} + \|\bm{F}\|_{L^6(\Omega)}),
    \end{align*}
    which completes the proof.
\end{proof}

For a fixed $u \in H_0^1(\Omega)$, we define the bilinear form
\begin{align*}
    a_u(v,w) := \int_{\Omega} \left( \nabla v \cdot \nabla w + Vvw + \beta |u|^2 vw \right) \mathrm{d}x, \qquad v, w \in H_0^1(\Omega),
\end{align*}
which is an inner product on $H_0^1(\Omega)$, and we denote by $\|\cdot\|_{a_u}$ the induced norm.

Let $H \subset H_0^1(\Omega)$ be a nontrivial closed subspace, and let $u_H \in H \cap \M$ be a minimizer of $E$ on $H \cap \M$, i.e.,
\begin{align*}
    E(u_H) = \min_{v \in H \cap \M} E(v).
\end{align*}
Then $u_H$ satisfies the Euler--Lagrange equation
\begin{align}\label{equ: GPEVP on H}
    a_{u_H}(u_H,v) = \lambda_H (u_H,v)_{L^2(\Omega)} \qquad \forall\, v \in H,
\end{align}
where $\lambda_H := a_{u_H}(u_H,u_H)$ is the associated Lagrange multiplier.

\begin{lemma}\label{lem: energy minus on H}
    For every $u \in H \cap \M$, the following identity holds:
    \begin{align}\label{equ: energy difference}
        E(u) - E(u_H) = \frac12\big(a_{u_H}(u,u) - \lambda_H\big) + \frac{\beta}{4}\|u^2 - u_H^2\|_{L^2(\Omega)}^2.
    \end{align}
\end{lemma}
\begin{proof}
    By the definition of $\lambda_H$, we have
    \begin{align*}
        \lambda_H = \int_{\Omega} \big(|\nabla u_H|^2 + V u_H^2 + \beta u_H^4\big)\,\mathrm{d}x,
    \end{align*}
    and hence
    \begin{align*}
        E(u_H) = \frac{\lambda_H}{2} - \frac{\beta}{4}\int_{\Omega} u_H^4 \,\mathrm{d}x.
    \end{align*}
    Adding and subtracting $\frac{\beta}{2}\int_{\Omega} u_H^2 u^2 \,\mathrm{d}x$, we obtain
    \begin{align*}
        E(u) - E(u_H) =& \int_{\Omega} \frac12 |\nabla u|^2 + \frac12 Vu^2 + \frac{\beta}{2}u_H^2u^2 \, \text{d}x \\
        & + \int_{\Omega} \frac{\beta}{4}u^4 - \frac{\beta}{2}u_H^2u^2 + \frac{\beta}{4} u_H^4 \, \text{d}x - \frac{\lambda_H}{2} \\
         = &  \frac12(a_{u_H}(u,u) - \lambda_H) + \frac{\beta}{4} \int_{\Omega} (u^2 - u_H^2)^2 \, \text{d}x,
    \end{align*}
    which completes the proof.
\end{proof}

Consider the following linear eigenvalue problem on $H$: find $(w,\mu) \in H \times \R$ such that
\begin{align}\label{LEVP on H}
    (\nabla w, \nabla v)_{L^2(\Omega)} + (Vw,v)_{L^2(\Omega)} + \beta (|u_H|^2 w,v)_{L^2(\Omega)} = \mu (w,v)_{L^2(\Omega)} \qquad \forall\, v \in H.
\end{align}
It follows from \eqref{equ: GPEVP on H} that $\lambda_H$ is an eigenvalue of \eqref{LEVP on H} with corresponding eigenfunction $u_H$. Moreover, the following result holds.
\begin{lemma}\label{lem: technical lemma 1}
    Assume that $\lambda_H$ is the smallest eigenvalue of \eqref{LEVP on H} and that it is simple, and let $\mu_2 > \lambda_H$ denote the second smallest eigenvalue of \eqref{LEVP on H}. Then for every $u \in H \cap \M$,
    \begin{align}\label{equ: technical lemma 1}
        \min_{s \in \{1,-1\}} \|u - s u_H\|_{a_{u_H}}^2 \le \frac{2(\mu_2 + \lambda_H)}{\mu_2 - \lambda_H} \big(E(u) - E(u_H)\big).
    \end{align}
\end{lemma}
\begin{proof}
    Fix $u \in H \cap \M$, set $t := |(u,u_H)_{L^2(\Omega)}|$, and choose $s \in \{-1,1\}$ such that $(su, u_H)_{L^2(\Omega)} = t$. Let $\tilde{u} := su$ and $w := \tilde{u} - t u_H \in H$. Since $\|u\|_{L^2(\Omega)} = \|u_H\|_{L^2(\Omega)} = 1$, the Cauchy--Schwarz inequality yields $0 \le t \le 1$. Moreover,
    \begin{align*}
        (w,u_H)_{L^2(\Omega)} = (\tilde{u},u_H)_{L^2(\Omega)} - t\|u_H\|_{L^2(\Omega)}^2 = t - t = 0,
    \end{align*}
    and hence
    \begin{align*}
        1 = \|\tilde{u}\|_{L^2(\Omega)}^2 = \|w + t u_H\|_{L^2(\Omega)}^2 = \|w\|_{L^2(\Omega)}^2 + t^2.
    \end{align*}
    Taking $v = w \in H$ in \eqref{equ: GPEVP on H} and using $(w,u_H)_{L^2(\Omega)} = 0$, we obtain
    \begin{align*}
        a_{u_H}(u_H, w) = \lambda_H(u_H,w)_{L^2(\Omega)} = 0.
    \end{align*}
    Together with $a_{u_H}(u_H,u_H) = \lambda_H$ and $\|w\|_{L^2(\Omega)}^2 = 1 - t^2$, this yields
    \begin{align*}
        a_{u_H}(\tilde{u},\tilde{u}) - \lambda_H & = a_{u_H}(w + tu_H,w + tu_H) - \lambda_H \\
        & = a_{u_H}(w,w) + 2t\,a_{u_H}(w,u_H) + t^2a_{u_H}(u_H,u_H) - \lambda_H \\
        & = a_{u_H}(w,w) - \lambda_H(1 - t^2) = a_{u_H}(w,w) - \lambda_H \|w\|_{L^2(\Omega)}^2.
    \end{align*}
    Since $\lambda_H$ is the smallest eigenvalue of \eqref{LEVP on H} with eigenfunction $u_H$, the Courant--Fischer min-max principle~\cite{babuska1991eigenvalue} yields
    \begin{align*}
        \mu_2 = \inf_{0 \neq z \in H,\, (z,u_H)_{L^2(\Omega)} = 0} \frac{a_{u_H}(z,z)}{\|z\|_{L^2(\Omega)}^2},
    \end{align*}
    and hence $\|w\|_{L^2(\Omega)}^2 \le \mu_2^{-1} a_{u_H}(w,w)$. Consequently,
    \begin{align*}
        a_{u_H}(\tilde{u},\tilde{u}) - \lambda_H \ge \left(1 - \frac{\lambda_H}{\mu_2}\right)a_{u_H}(w,w).
    \end{align*}
    Since $E$ is even, i.e., $E(u) = E(\tilde{u})$, Lemma~\ref{lem: energy minus on H} together with the above estimate yields
    \begin{align*}
        E(u) - E(u_H) = E(\tilde{u}) - E(u_H) \ge \frac12\big(a_{u_H}(\tilde{u},\tilde{u}) - \lambda_H\big) \ge \frac{1}{2}\left(1 - \frac{\lambda_H}{\mu_2}\right)a_{u_H}(w,w).
    \end{align*}
    
    On the other hand, using $a_{u_H}(u_H,w) = 0$ and $a_{u_H}(u_H,u_H) = \lambda_H$, we obtain
    \begin{align*}
        \|\tilde{u} - u_H\|_{a_{u_H}}^2 = a_{u_H}(w,w) + \lambda_H(1-t)^2.
    \end{align*}
    Since $0 \le t \le 1$, we have $(1-t)^2 \le 1 - t^2 = \|w\|_{L^2(\Omega)}^2$. Together with $\lambda_H > 0$ and $\|w\|_{L^2(\Omega)}^2 \le \mu_2^{-1} a_{u_H}(w,w)$, this yields
    \begin{align*}
        \|\tilde{u} - u_H\|_{a_{u_H}}^2 = a_{u_H}(w,w) + \lambda_H(1-t)^2 \le \left(1 + \frac{\lambda_H}{\mu_2}\right)a_{u_H}(w,w).
    \end{align*}
    Combining this with the lower bound for $E(u) - E(u_H)$ above, we arrive at
    \begin{align*}
        \|\tilde{u} - u_H\|_{a_{u_H}}^2 \le \frac{2(\mu_2 + \lambda_H)}{\mu_2 - \lambda_H} \big( E(u) - E(u_H) \big).
    \end{align*}
    Since $\|\tilde{u} - u_H\|_{a_{u_H}} = \|u - s u_H\|_{a_{u_H}}$, this completes the proof.
\end{proof}

\begin{coro}\label{coro: uniqueness on H}
    Under the assumptions of Lemma~\ref{lem: technical lemma 1}, the minimizer $u_H$ of $E$ on $H \cap \M$ is unique up to sign, i.e., every minimizer $u$ of $E$ on $H \cap \M$ satisfies $u = u_H$ or $u = -u_H$.
\end{coro}

\section{The standard finite element method}\label{sec: standard FEM}

A discrete ground state of the standard FEM is defined as
\begin{align*}
    u_h^* \in \argmin_{v_h \in V_h^0 \cap \M} E(v_h) \qquad \text{with} \qquad (u_h^*,u^*)_{L^2(\Omega)} \ge 0.
\end{align*}
Taking $H = V_h^0$ in \eqref{equ: GPEVP on H}, we see that $u_h^*$ satisfies the discrete Euler--Lagrange equation
\begin{align}\label{equ: discrete GPEVP}
    a_{u_h^*}(u_h^*,v_h) = \lambda_h^*(u_h^*,v_h)_{L^2(\Omega)} \qquad \forall\, v_h \in V_h^0,
\end{align}
where $\lambda_h^* := a_{u_h^*}(u_h^*,u_h^*)$ is the associated Lagrange multiplier.

\subsection{Uniqueness of the discrete ground state}

Since $V_h^0 \cap \M$ is compact in the finite-dimensional space $V_h^0$ and $E$ is continuous, a discrete ground state $u_h^*$ exists. Moreover, with the sign chosen as above, it converges to the ground state $u^*$ in $H^1(\Omega)$; see, e.g., \cite{cances2010numerical,chen2011finite,zhou2004analysis}:
\begin{align}\label{equ: convergence of standard FEM}
    \|u_h^* - u^*\|_{H^1(\Omega)} \to 0 \qquad \text{as } h \to 0.
\end{align}
As a consequence, the bilinear forms $a_{u_h^*}(\cdot,\cdot)$ and $a_{u^*}(\cdot,\cdot)$ are equivalent, uniformly in $h$, for sufficiently small $h$.
\begin{lemma}\label{lem: equivalence of a}
    There exists a sequence $\epsilon_h \ge 0$ with $\epsilon_h \to 0$ as $h \to 0$ such that for all $v \in H_0^1(\Omega)$,
    \begin{align*}
        (1 - \epsilon_h)\, a_{u^*}(v,v) \le a_{u_h^*}(v,v) \le (1 + \epsilon_h)\, a_{u^*}(v,v).
    \end{align*}
\end{lemma}
\begin{proof}
    For any $v \in H_0^1(\Omega)$, it follows from H\"older's inequality and the Sobolev embedding $H_0^1(\Omega) \hookrightarrow L^6(\Omega)$ that
    \begin{align*}
        |a_{u_h^*}(v,v) - a_{u^*}(v,v)| &= \beta \Big|\int_\Omega (u_h^* + u^*)(u_h^* - u^*) v^2 \,\mathrm{d}x\Big| \\
        &\le C\beta \|u_h^* + u^*\|_{H^1(\Omega)}\|u_h^* - u^*\|_{H^1(\Omega)}\|\nabla v\|_{L^2(\Omega)}^2.
    \end{align*}
    Setting $\epsilon_h := C\beta \|u_h^* + u^*\|_{H^1(\Omega)}\|u_h^* - u^*\|_{H^1(\Omega)}$ and using $\|\nabla v\|_{L^2(\Omega)}^2 \le a_{u^*}(v,v)$, we obtain
    \begin{align*}
        |a_{u_h^*}(v,v) - a_{u^*}(v,v)| \le \epsilon_h\, a_{u^*}(v,v).
    \end{align*}
    Since $\|u_h^*\|_{H^1(\Omega)}$ is uniformly bounded, \eqref{equ: convergence of standard FEM} implies $\epsilon_h \to 0$ as $h \to 0$, which completes the proof.
\end{proof}

Let $\Lambda_{1,h} \le \Lambda_{2,h} \le \cdots$ denote the eigenvalues of the following linear eigenvalue problem: find $(w_h, \Lambda_h) \in V_h^0 \times \R$ with $w_h \neq 0$ such that
\begin{align*}
    a_{u_h^*}(w_h,v_h) = \Lambda_h(w_h,v_h)_{L^2(\Omega)} \qquad \forall\, v_h \in V_h^0.
\end{align*}
It follows from \eqref{equ: discrete GPEVP} that $(\lambda_h^*, u_h^*)$ is an eigenpair of this problem. Moreover, the following result holds.
\begin{lemma}\label{lem: discrete eigengap}
    There exists $h_0 > 0$ such that for all $0<h \le h_0$,
    \begin{align*}
        \Lambda_{1,h} = \lambda_h^* < \Lambda_{2,h}.
    \end{align*}
\end{lemma}
\begin{proof}
    For $v \in H_0^1(\Omega) \setminus \{0\}$, we define the Rayleigh quotients
    \begin{align*}
        R_{u^*}(v) := \frac{a_{u^*}(v,v)}{\|v\|_{L^2(\Omega)}^2}, \qquad R_{u_h^*}(v) := \frac{a_{u_h^*}(v,v)}{\|v\|_{L^2(\Omega)}^2}.
    \end{align*}
    For a subspace $X \subset H_0^1(\Omega)$, let $\mathcal{S}_2(X)$ denote the set of all two-dimensional subspaces of $X$. Let $\mu_2$ denote the second smallest eigenvalue of \eqref{equ: LEVP}. By Proposition~\ref{prop: simple eigenvalue}, we have $\mu_2 > \mu_1 = \lambda^*$, and the min-max principle yields
    \begin{align*}
        \mu_2 = \inf_{Y \in \mathcal{S}_2(H_0^1(\Omega))} \max_{0 \neq v \in Y} R_{u^*}(v).
    \end{align*}
    Fix $S_h \in \mathcal{S}_2(V_h^0)$. Since $S_h$ is finite-dimensional, there exists $\hat{v}_h \in S_h \setminus \{0\}$ such that
    \begin{align*}
        R_{u^*}(\hat{v}_h) = \max_{0 \neq v_h \in S_h} R_{u^*}(v_h).
    \end{align*}
    By Lemma~\ref{lem: equivalence of a}, we obtain
    \begin{align*}
        \max_{0 \neq v_h \in S_h} R_{u_h^*}(v_h) \ge R_{u_h^*}(\hat{v}_h) \ge (1 - \epsilon_h) R_{u^*}(\hat{v}_h) = (1 - \epsilon_h)\max_{0 \neq v_h \in S_h} R_{u^*}(v_h).
    \end{align*}
    Taking the infimum over all $S_h \in \mathcal{S}_2(V_h^0)$ and using the min-max principle for the discrete problem, we arrive at
    \begin{align*}
        \Lambda_{2,h} = \inf_{S_h \in \mathcal{S}_2(V_h^0)}\max_{0 \neq v_h \in S_h} R_{u_h^*}(v_h) \ge (1 - \epsilon_h)\underbrace{\inf_{S_h \in \mathcal{S}_2(V_h^0)}\max_{0 \neq v_h \in S_h} R_{u^*}(v_h)}_{=: \hat{\Lambda}_{2,h}}.
    \end{align*}
    Since $V_h^0 \subset H_0^1(\Omega)$, we have $\mathcal{S}_2(V_h^0) \subset \mathcal{S}_2(H_0^1(\Omega))$, and hence
\begin{align*}
    \hat{\Lambda}_{2,h} = \inf_{S_h \in \mathcal{S}_2(V_h^0)}\max_{0 \neq v_h \in S_h} R_{u^*}(v_h) \ge \inf_{Y \in \mathcal{S}_2(H_0^1(\Omega))}\max_{0 \neq v \in Y} R_{u^*}(v) = \mu_2.
\end{align*}
Consequently,
\begin{align*}
    \Lambda_{2,h} \ge (1 - \epsilon_h)\mu_2.
\end{align*}

    Set $\delta := \mu_2 - \lambda^* > 0$. Since $\lambda_h^* = a_{u_h^*}(u_h^*,u_h^*)$ and $\lambda^* = a_{u^*}(u^*,u^*)$, it follows from \eqref{equ: convergence of standard FEM} that $\lambda_h^* \to \lambda^*$ as $h \to 0$. Together with $\epsilon_h \to 0$, there exists $h_0 > 0$ such that for all $h \le h_0$,
\begin{align*}
    \epsilon_h \mu_2 < \frac{\delta}{4}, \qquad |\lambda_h^* - \lambda^*| < \frac{\delta}{4}.
\end{align*}
For such $h$, we obtain
\begin{align*}
    \Lambda_{2,h} - \lambda_h^* \ge \Big(\mu_2 - \frac{\delta}{4}\Big) - \Big(\lambda^* + \frac{\delta}{4}\Big) = \frac{\delta}{2} > 0.
\end{align*}
Since $\lambda_h^*$ is an eigenvalue of the discrete linear problem and $\lambda_h^* < \Lambda_{2,h}$, it must be the smallest one, i.e., $\Lambda_{1,h} = \lambda_h^* < \Lambda_{2,h}$. This completes the proof.
\end{proof}

\begin{coro}\label{coro: uniform gap}
    It holds that
    \begin{align*}
        \liminf_{h \to 0} (\Lambda_{2,h} - \lambda_h^*) \ge \mu_2 - \lambda^* > 0.
    \end{align*}
\end{coro}
This follows directly from the proof of Lemma~\ref{lem: discrete eigengap}, since $\Lambda_{2,h} - \lambda_h^* \ge (1 - \epsilon_h)\mu_2 - \lambda_h^*$, $\epsilon_h \to 0$, and $\lambda_h^* \to \lambda^*$.
\begin{remark}
    In \cite[Section~2]{hassan2025relation}, it is pointed out that, if the discrete ground states converge to $u^*$ in $L^\infty(\Omega)$, a min-max argument shows that $\lambda_h^*$ is the smallest eigenvalue of the discrete linear eigenvalue problem for sufficiently small $h$. Lemma~\ref{lem: discrete eigengap} shows that convergence in $H^1(\Omega)$ is sufficient, and that the smallest eigenvalue is moreover separated from the rest of the spectrum by a uniform gap.
\end{remark}

Combining Lemma~\ref{lem: discrete eigengap} with Corollary~\ref{coro: uniqueness on H} for $H = V_h^0$, we obtain the following result.

\begin{theorem}\label{thm: uniqueness of discrete ground state}
    There exists $h_0 > 0$ such that for all $0<h \le h_0$, the minimizer of $E$ on $V_h^0 \cap \M$ is unique up to sign, i.e., $u_h^*$ and $-u_h^*$ are the only two minimizers. 
\end{theorem}

\subsection{\texorpdfstring{\boldmath $L^p$}{Lp} and \texorpdfstring{\boldmath $W^{1,p}$}{W1p} convergence}

We first recall the following a priori error estimates in the $L^2$- and $H^1$-norms; see, e.g., \cite{cances2010numerical,chen2011finite,zhou2004analysis}.
\begin{theorem}\label{thm: L^2 and H^1 estimates for standard FEM}
    There exists $h_0 > 0$ such that for all $0 < h \le h_0$,
    \begin{align*}
        |\lambda_h^* - \lambda^*| + \|u_h^* - u^*\|_{L^2(\Omega)} + h \|\nabla (u_h^* - u^*)\|_{L^2(\Omega)} \le Ch^2.
    \end{align*}
\end{theorem}

As a consequence, we obtain the following uniform bounds.
\begin{lemma}\label{lem: Linfty bound for standard FEM}
    There exists $h_0 > 0$ such that for all $0 < h \le h_0$,
    \begin{align*}
        \|u_h^*\|_{L^{\infty}(\Omega)} + |\lambda_h^*| \le C.
    \end{align*}
\end{lemma}
\begin{proof}
    The bound for $\lambda_h^*$ follows directly from Theorem~\ref{thm: L^2 and H^1 estimates for standard FEM}, so it remains to bound $u_h^*$. By the inverse inequality~\cite{brenner2008mathematical} and the interpolation estimate, we obtain
    \begin{align*}
        \|u_h^*\|_{L^{\infty}(\Omega)} & \le \|u_h^* - I_h u^*\|_{L^{\infty}(\Omega)} + \|I_h u^*\|_{L^{\infty}(\Omega)} \\
        & \le Ch^{-\frac{d}{2}}\big(\|u_h^* - u^*\|_{L^2(\Omega)} + \|u^* - I_hu^*\|_{L^2(\Omega)}\big) + \|u^*\|_{L^{\infty}(\Omega)} \\
        & \le C h^{2 - \frac{d}{2}} + C \le C,
    \end{align*}
    where we used $d \le 3$ in the last step.
\end{proof}

Since $\Omega$ is bounded, H\"older's inequality and Theorem~\ref{thm: L^2 and H^1 estimates for standard FEM} imply that for all $p \in [1,2]$,
\begin{align}\label{equ: Lp estimate p le 2}
    \|u_h^* - u^*\|_{L^p(\Omega)} + h \|\nabla (u_h^* - u^*)\|_{L^p(\Omega)} \le Ch^2.
\end{align}
It therefore remains to consider the case $2 < p \le \infty$.

We prove the convergence in $L^p$ and $W^{1,p}$ by a perturbation argument; see, e.g., \cite{dai2008convergence,chen2011finite,he2011convergence}. The key observation is that both $u^*$ and $u_h^*$ can be regarded as solutions of Poisson problems. More precisely, by \eqref{equ: GPEVP}, the ground state $u^* \in H_0^1(\Omega)$ satisfies
\begin{align}\label{equ: perturbation equation}
    - \Delta u^* = f := \lambda^* u^* - V u^* - \beta (u^*)^3,
\end{align}
and, by \eqref{equ: discrete GPEVP}, the discrete ground state $u_h^*$ satisfies
\begin{align*}
    (\nabla u_h^*, \nabla v_h)_{L^2(\Omega)} = (f_h, v_h)_{L^2(\Omega)} \qquad \forall\, v_h \in V_h^0,
\end{align*}
where $f_h := \lambda_h^* u_h^* - Vu_h^* - \beta (u_h^*)^3$. We first estimate the difference $f - f_h$.
\begin{lemma}\label{lem: f-f_h}
    There exists $h_0 > 0$ such that for all $0 < h \le h_0$,
    \begin{align*}
        \|f - f_h\|_{L^2(\Omega)} \le C h^2.
    \end{align*}
\end{lemma}
\begin{proof}
    We decompose $f - f_h$ as
    \begin{align*}
        f - f_h & = (\lambda^* - \lambda_h^*)u^* + \lambda_h^*(u^* - u_h^*) - V(u^* - u_h^*) - \beta \big((u^*)^3 - (u_h^*)^3\big) \\
        & = (\lambda^* - \lambda_h^*)u^* + \Big( \lambda_h^* - V - \beta \big((u^*)^2 + u^*u_h^* + (u_h^*)^2\big)\Big) (u^* - u_h^*).
    \end{align*}
    By Theorem~\ref{thm: L^2 and H^1 estimates for standard FEM}, Lemma~\ref{lem: Linfty bound for standard FEM}, and $\|u^*\|_{L^2(\Omega)} = 1$, we obtain
    \begin{align*}
        \|f - f_h\|_{L^2(\Omega)}
        &\le \Big\| \lambda_h^* - V - \beta \big((u^*)^2 + u^*u_h^* + (u_h^*)^2\big)\Big\|_{L^{\infty}(\Omega)} \\
        &\quad \times \|u^* - u_h^*\|_{L^2(\Omega)} + |\lambda^* - \lambda_h^*| \le Ch^2,
    \end{align*}
    which completes the proof.
\end{proof}

The error admits the decomposition
\begin{align}\label{equ: error decomposition}
    u_h^* - u^* = (u_h^* - R_hu^*) + (R_hu^* - u^*) =: \theta_h + \rho_h.
\end{align}
By the definition of $R_h$ and \eqref{equ: perturbation equation}, $\theta_h \in V_h^0$ satisfies
\begin{align*}
    (\nabla \theta_h,\nabla v_h)_{L^2(\Omega)} = (f_h - f, v_h)_{L^2(\Omega)} \qquad \forall\, v_h \in V_h^0.
\end{align*}
Lemmas~\ref{lem: general framework} and~\ref{lem: f-f_h} then yield the following superconvergence result.
\begin{lemma}\label{lem: superconvergence le 6}
    There exists $h_0 > 0$ such that for all $0 < h \le h_0$,
    \begin{align*}
        \|\nabla \theta_h\|_{L^6(\Omega)} \le Ch^2.
    \end{align*}
\end{lemma}

\begin{remark}
    For the linear finite element method, it is shown in \cite[Theorem~2]{hassan2025relation} that
    \begin{align*}
        \|\Pi_h u^* - u_h^*\|_{H^1(\Omega)} \le Ch^2,
    \end{align*}
    where $\Pi_h$ denotes the $H^1(\Omega)$-orthogonal projection onto $V_h^0$. Lemma~\ref{lem: superconvergence le 6} provides an analogous estimate for the Ritz projection in the stronger $W^{1,6}(\Omega)$-norm, based on a different argument.
\end{remark}

Combining Lemma~\ref{lem: superconvergence le 6} with the Sobolev embedding $W_0^{1,6}(\Omega) \hookrightarrow L^{\infty}(\Omega)$ for $d \le 3$ and the inverse inequality, one obtains the following estimates.
\begin{lemma}\label{lem: Linfty estimates of theta_h}
    There exists $h_0 > 0$ such that for all $0 < h \le h_0$,
    \begin{align*}
        \|\theta_h\|_{L^{\infty}(\Omega)} \le Ch^2, \qquad \|\nabla \theta_h\|_{L^{\infty}(\Omega)} \le Ch^{2 - \frac{d}{6}}.
    \end{align*}
\end{lemma}

To establish convergence in $L^p(\Omega)$ and $W^{1,p}(\Omega)$, particularly for $p=\infty$, additional regularity of $u^*$ is required. Since $V \in L^\infty(\Omega)$ and $u^* \in H^2(\Omega) \hookrightarrow L^\infty(\Omega)$ for $d \le 3$, the right-hand side $f$ of \eqref{equ: perturbation equation} belongs to $L^\infty(\Omega)$. Elliptic regularity \cite{Dauge1988,grisvard2011elliptic} then yields $u^* \in W^{2,p}(\Omega) \cap H_0^1(\Omega)$ for every $2 \le p < \infty$. Combining Proposition~\ref{prop: p convergence of Ritz} with Lemma~\ref{lem: Linfty estimates of theta_h}, we arrive at the following result.
\begin{theorem}\label{thm: p convergence}
    For every $1 \le p < \infty$, there exists $h_0 > 0$ such that for all $0 < h \le h_0$,
    \begin{align*}
        \|u^* - u_h^*\|_{L^p(\Omega)} + h \|\nabla (u^* - u_h^*)\|_{L^p(\Omega)} \le C_ph^2.
    \end{align*}
\end{theorem}
\begin{proof}
    In view of \eqref{equ: Lp estimate p le 2}, it suffices to consider $2 < p < \infty$. Since $u^* \in W^{2,p}(\Omega)$, Proposition~\ref{prop: p convergence of Ritz} and Lemma~\ref{lem: Linfty estimates of theta_h} yield
    \begin{align*}
        \|u^* - u_h^*\|_{L^p(\Omega)} \le \|\theta_h\|_{L^p(\Omega)} + \|\rho_h\|_{L^p(\Omega)} \le C \|\theta_h\|_{L^{\infty}(\Omega)} + C_p h^2 \le C_ph^2.
    \end{align*}
    Similarly, using $2 - \frac{d}{6} \ge 1$ for $d \le 3$, we obtain
    \begin{align*}
        \|\nabla (u^* - u_h^*)\|_{L^p(\Omega)} \le C \|\nabla \theta_h\|_{L^{\infty}(\Omega)} + \|\nabla \rho_h\|_{L^{p}(\Omega)} \le C h^{2 - \frac{d}{6}} + C_ph \le C_ph,
    \end{align*}
    which completes the proof.
\end{proof}
\begin{remark}
    The superconvergence result in \cite{hassan2025relation} implies $\|\nabla \theta_h\|_{L^2(\Omega)} \le Ch^2$. For $d = 2$, the embedding $H_0^1(\Omega) \hookrightarrow L^p(\Omega)$ for every $1 \le p < \infty$ and the inverse inequality then give
    \begin{align*}
        \|\theta_h\|_{L^p(\Omega)} \le C_p \|\nabla \theta_h\|_{L^2(\Omega)} \le C_p h^2, \qquad \|\nabla \theta_h\|_{L^{\infty}(\Omega)} \le C h^{-1}\|\nabla \theta_h\|_{L^{2}(\Omega)} \le C h.
    \end{align*}
    Combined with Proposition~\ref{prop: p convergence of Ritz}, this recovers Theorem~\ref{thm: p convergence} for $d = 2$. For $d = 3$, however, this argument only applies to $p \le 6$, since $H_0^1(\Omega) \hookrightarrow L^p(\Omega)$ holds only for $p \le 6$ and the inverse inequality yields $\|\nabla \theta_h\|_{L^p(\Omega)} \le Ch^{\frac12 + \frac3p}$, which is suboptimal for $p > 6$.
\end{remark}

To establish convergence in $L^{\infty}(\Omega)$ and $W^{1,\infty}(\Omega)$, we need the $W^{2,\infty}(\Omega)$-regularity of $u^*$. Since $u^* \in H^2(\Omega) \hookrightarrow C(\overline{\Omega})$ and $u^*$ vanishes on $\partial\Omega$, the right-hand side $f$ of \eqref{equ: perturbation equation} vanishes on $\partial\Omega$, in particular at the corners (and, for $d=3$, along the edges) of $\Omega$. This compatibility condition, combined with additional regularity of $V$, excludes the corner singularities and leads to the following higher regularity result; see \cite{Dauge1988,grisvard2011elliptic,hell2014compatibility}.
\begin{prop}\label{prop: W2infty regularity}
    Assume, in addition, that $V \in C^{0,\alpha}(\overline{\Omega})$ for some $\alpha \in (0,1)$. Then $u^* \in W^{2,\infty}(\Omega)$.
\end{prop}

Proceeding as in the proof of Theorem~\ref{thm: p convergence}, we obtain the following result.
\begin{theorem}\label{thm: Linfty and W1infty convergence}
    Assume that $V \in C^{0,\alpha}(\overline{\Omega})$ for some $\alpha \in (0,1)$. Then there exists $h_0 > 0$ such that for all $0 < h \le h_0$,
    \begin{align*}
        \|u^* - u_h^*\|_{L^{\infty}(\Omega)} \le C h^2 |\ln h|, \qquad \|\nabla (u^* - u_h^*)\|_{L^{\infty}(\Omega)} \le Ch.
    \end{align*}
\end{theorem}

\subsection{Superconvergence}

Lemma~\ref{lem: superconvergence le 6} shows that $\theta_h$ superconverges in $W^{1,p}(\Omega)$ for $1 \le p \le 6$. With the help of Theorem~\ref{thm: p convergence}, this result can be extended to $1 \le p \le \infty$.

\begin{theorem}\label{thm: superconvergence of standard FEM}
    For every $1 \le p \le \infty$, there exists $h_0 > 0$ such that for all $0 < h \le h_0$,
    \begin{align*}
        \|\nabla(R_hu^* - u_h^*)\|_{L^p(\Omega)} \le C_ph^2.
    \end{align*}
\end{theorem}
\begin{proof}
    It suffices to consider $p = \infty$. The argument of Lemma~\ref{lem: f-f_h}, combined with the $L^4$-estimate in Theorem~\ref{thm: p convergence}, yields
    \begin{align*}
        \|f - f_h\|_{L^4(\Omega)} \le Ch^2.
    \end{align*}
    Let $w \in H_0^1(\Omega)$ be the solution of $-\Delta w = f_h - f$. By elliptic regularity, $w \in W^{2,4}(\Omega)$ with $\|w\|_{W^{2,4}(\Omega)} \le C\|f - f_h\|_{L^4(\Omega)}$. Moreover, $\theta_h = R_h w$ by the definitions of $\theta_h$ and $R_h$. Hence, Proposition~\ref{prop: stability of the Ritz projection} and the Sobolev embedding $W^{2,4}(\Omega) \hookrightarrow W^{1,\infty}(\Omega)$ for $d \le 3$ yield
    \begin{align*}
        \|\nabla \theta_h\|_{L^{\infty}(\Omega)} \le C \|\nabla w\|_{L^{\infty}(\Omega)} \le C \|w\|_{W^{2,4}(\Omega)} \le C \|f - f_h\|_{L^4(\Omega)} \le Ch^2,
    \end{align*}
    which completes the proof.
\end{proof}

\section{The mass-lumped finite element method}\label{sec: mass-lumped FEM}

The mass-lumped FEM differs from the standard FEM in that the $L^2$ inner product is replaced by a discrete one, obtained by applying the vertex quadrature rule elementwise. More precisely, for an element $K \in \mathcal{T}_h$ with vertices $\mathcal{N}_K = \{ x_{K,i} \}_{i=0}^d$, the integral of $\phi \in C(\overline{K})$ over $K$ is approximated by
\begin{align*}
    Q_K (\phi) := \frac{|K|}{d+1} \sum_{i=0}^d \phi(x_{K,i}).
\end{align*}
This quadrature rule is exact for affine functions on $K$. Accordingly, for $\phi, \psi \in C(\overline{K})$, we define
\begin{align*}
    (\phi,\psi)_{\ell^2,K} := Q_K (\phi\psi) = \frac{|K|}{d+1} \sum_{i=0}^d \phi(x_{K,i})\psi(x_{K,i}).
\end{align*}
This leads to the global discrete $L^2$ inner product and discrete $L^2$ norm
\begin{align*}
    (\phi,\psi)_{\ell^2} := \sum_{K \in \mathcal{T}_h} (\phi,\psi)_{\ell^2,K} = \sum_{z \in \mathcal{N}_h} \omega_z \phi(z)\psi(z), \qquad \|\phi\|_{\ell^2}^2 := (\phi,\phi)_{\ell^2},
\end{align*}
where $\omega_z := \frac{1}{d+1}\sum_{K \in \mathcal{T}_h,\, z \in \mathcal{N}_K} |K|$.

Throughout this section, we additionally assume that $V \in W^{2,\infty}(\Omega)$. For $v_h \in V_h^0$, the discrete energy of the mass-lumped FEM is defined as
\begin{align*}
    E_h(v_h) := \frac12 \|\nabla v_h\|_{L^2(\Omega)}^2 + \frac12 (V v_h, v_h)_{\ell^2} + \frac{\beta}{4} (v_h^2,v_h^2)_{\ell^2}.
\end{align*}
The discrete ground state of the mass-lumped FEM is then defined as
\begin{align}\label{equ: mass-lumped ground state}
    u_h^{*,\ell} \in \argmin_{v_h \in V_h^0,\, \|v_h\|_{\ell^2} = 1} E_h(v_h) \qquad \text{with} \qquad (u_h^{*,\ell}, u^*)_{L^2(\Omega)} \ge 0.
\end{align}
It satisfies the discrete Euler--Lagrange equation
\begin{align}\label{equ: mass-lumped discrete GPEVP}
    (\nabla u_h^{*,\ell}, \nabla v_h)_{L^2(\Omega)} + (V u_h^{*,\ell},v_h)_{\ell^2} + \beta \big((u_h^{*,\ell})^3,v_h\big)_{\ell^2} = \lambda_h^{*,\ell}(u_h^{*,\ell},v_h)_{\ell^2} \qquad \forall\, v_h \in V_h^0,
\end{align}
where $\lambda_h^{*,\ell} := 2E_h(u_h^{*,\ell}) + \frac{\beta}{2}\big((u_h^{*,\ell})^2,(u_h^{*,\ell})^2\big)_{\ell^2}$ is the associated Lagrange multiplier.
The following result is established in \cite{hauck2026positivity}.
\begin{prop}\label{prop: mass-lumped FEM convergence}
    Assume that the stiffness matrix associated with $V_h^0$ is an irreducible M-matrix. Then the discrete ground state $u_h^{*,\ell}$ is unique and positive in $\Omega$. Moreover, there exists $h_0 > 0$ such that for all $0 < h \le h_0$,
    \begin{align*}
        |\lambda^* - \lambda_h^{*,\ell}| + \| u_h^{*,\ell} - u^*\|_{L^2(\Omega)} + h\|\nabla (u_h^{*,\ell} - u^*)\|_{L^2(\Omega)} \le C h^2
    \end{align*}
    and
    \begin{align*}
        \|u_h^{*,\ell}\|_{L^{\infty}(\Omega)} + |\lambda_h^{*,\ell}| \le C.
    \end{align*}
\end{prop}

\subsection{Some technical results}

For an element $K \in \mathcal{T}_h$, let $\varphi_{K,0},\dots,\varphi_{K,d}$ denote the barycentric coordinates of $K$, i.e., the affine functions on $K$ satisfying $\varphi_{K,i}(x_{K,j}) = \delta_{ij}$. Then
\begin{align*}
    \int_{K} \varphi_{K,i}\varphi_{K,j} \,\mathrm{d}x = \frac{|K|}{(d+1)(d+2)}(1 + \delta_{ij}), \qquad i,j = 0,\dots,d.
\end{align*}
Moreover, for $1 \le q < \infty$, we define the discrete norms on $V_h$ by
\begin{align*}
    \|v_h\|_{\ell^q} := \Big(\sum_{z \in \mathcal{N}_h} \omega_z |v_h(z)|^q\Big)^{\frac1q}, \qquad v_h \in V_h.
\end{align*}

\begin{lemma}\label{lem: norm equivalence}
    Let $1 \le q < \infty$. There exist constants $0 < c_q \le C_q$ such that for all $v_h \in V_h$,
    \begin{align*}
        c_q\|v_h\|_{L^q(\Omega)} \le \|v_h\|_{\ell^q} \le C_q\|v_h\|_{L^q(\Omega)}.
    \end{align*}
    In particular, for $q = 2$,
    \begin{align*}
        \|v_h\|_{L^2(\Omega)}^2 \le \|v_h\|_{\ell^2}^2 \le (d+2) \|v_h\|_{L^2(\Omega)}^2.
    \end{align*}
\end{lemma}
\begin{proof}
    The general case follows from the equivalence of norms on finite-dimensional spaces and a standard scaling argument. It remains to derive the explicit constants for $q = 2$. Fix $v_h \in V_h$ and $K \in \mathcal{T}_h$, and write $v_h|_K = \sum_{i=0}^d v_i \varphi_{K,i}$ with $v_i := v_h(x_{K,i})$. Then
\begin{align*}
    \|v_h\|_{L^2(K)}^2 = \frac{|K|}{(d+1)(d+2)}\left[ \Big( \sum_{i=0}^d v_i \Big)^2 + \sum_{i=0}^d v_i^2 \right],
    \qquad
    Q_K(v_h^2) = \frac{|K|}{d+1} \sum_{i=0}^d v_i^2.
\end{align*}
Since $0 \le \big(\sum_{i=0}^d v_i\big)^2 \le (d+1)\sum_{i=0}^d v_i^2$ by the Cauchy--Schwarz inequality, it follows that
\begin{align*}
    \|v_h\|_{L^2(K)}^2 \le Q_K(v_h^2) \le (d+2)\|v_h\|_{L^2(K)}^2.
\end{align*}
Summing over all $K \in \mathcal{T}_h$ yields the assertion for $q = 2$.
\end{proof}


For $K \in \mathcal{T}_h$, define the symmetric matrix
\begin{align*}
    D_K := \frac{1}{(d+1)(d+2)} \sum_{0 \le i < j \le d} (x_{K,i} - x_{K,j})(x_{K,i} - x_{K,j})^{T} \in \R^{d \times d}.
\end{align*}
\begin{lemma}\label{lem: consistency}
    Let $K \in \mathcal{T}_h$ and $v_h, w_h \in \mathbb{P}_1(K)$, and set $v_i := v_h(x_{K,i})$ and $w_i := w_h(x_{K,i})$ for $i = 0,\dots,d$. Then
    \begin{align}\label{equ: quadrature error nodal}
        \begin{aligned}
            Q_{K}(v_hw_h) - \int_{K}v_hw_h \,\mathrm{d}x
            &= \frac{|K|}{(d+1)(d+2)} \sum_{0 \le i < j \le d} (v_i - v_j)(w_i - w_j) \\
            &= \int_K (D_K \nabla v_h) \cdot \nabla w_h \,\mathrm{d}x.
        \end{aligned}
    \end{align}
    Moreover, $D_K$ is positive semidefinite and satisfies
    \begin{align*}
        0 \le \xi^{T} D_K \xi \le \frac{d}{2(d+2)}h_K^2 |\xi|^2 \qquad \forall\, \xi \in \R^d.
    \end{align*}
\end{lemma}
\begin{proof}
    Since $v_h|_K = \sum_{i=0}^d v_i \varphi_{K,i}$ and $w_h|_K = \sum_{j=0}^d w_j \varphi_{K,j}$, the integration formula for $\int_K \varphi_{K,i}\varphi_{K,j}\,\mathrm{d}x$ gives
    \begin{align*}
        \int_K v_h w_h \,\mathrm{d}x
        &= \frac{|K|}{(d+1)(d+2)} \sum_{i,j = 0}^d v_iw_j (1 + \delta_{ij})\\
        &= \frac{|K|}{(d+1)(d+2)} \bigg[ \sum_{i,j=0}^d v_iw_j + \sum_{i = 0}^d v_iw_i \bigg].
    \end{align*}
    Together with $Q_K(v_hw_h) = \frac{|K|}{d+1}\sum_{i=0}^d v_iw_i$, this yields
    \begin{align*}
        Q_{K}(v_hw_h) - \int_{K}v_hw_h \,\mathrm{d}x = \frac{|K|}{(d+1)(d+2)} \bigg[ (d+1)\sum_{i = 0}^d v_iw_i - \sum_{i,j=0}^d v_iw_j \bigg].
    \end{align*}
    On the other hand, a direct computation shows that
    \begin{align*}
        \sum_{0 \le i < j \le d} (v_i - v_j)(w_i - w_j)
        = \frac12 \sum_{i,j=0}^d (v_i - v_j)(w_i - w_j)
        = (d+1)\sum_{i = 0}^d v_iw_i - \sum_{i,j=0}^d v_iw_j.
    \end{align*}
    Since $v_h, w_h \in \mathbb{P}_1(K)$, their gradients are constant on $K$, and
\begin{align*}
    v_i - v_j = (x_{K,i} - x_{K,j})^T\nabla v_h, \qquad w_i - w_j = (x_{K,i} - x_{K,j})^T\nabla w_h.
\end{align*}
Hence,
\begin{align*}
    \sum_{0 \le i < j \le d} (v_i - v_j)(w_i - w_j)
    &= \sum_{0 \le i < j \le d} (\nabla v_h)^T(x_{K,i} - x_{K,j})(x_{K,i} - x_{K,j})^T\nabla w_h \\
    &= (d+1)(d+2)\, (\nabla v_h)^T D_K \nabla w_h.
\end{align*}
Multiplying by $\frac{|K|}{(d+1)(d+2)}$ and using that $\nabla v_h$ and $\nabla w_h$ are constant on $K$, we obtain
\begin{align*}
    Q_{K}(v_hw_h) - \int_{K}v_hw_h \,\mathrm{d}x = |K|\,(\nabla v_h)^T D_K \nabla w_h = \int_K (D_K \nabla v_h) \cdot \nabla w_h \,\mathrm{d}x.
\end{align*}
   Finally, for any $\xi \in \R^d$, the definition of $D_K$ gives
\begin{align*}
    \xi^T D_K \xi = \frac{1}{(d+1)(d+2)} \sum_{0 \le i < j \le d} \big( \xi \cdot (x_{K,i} - x_{K,j}) \big)^2.
\end{align*}
Since $|x_{K,i} - x_{K,j}| \le h_K$ and there are $\frac{d(d+1)}{2}$ pairs $(i,j)$ with $0 \le i < j \le d$, it follows that
\begin{align*}
    0 \le \xi^T D_K \xi \le \frac{1}{(d+1)(d+2)} \cdot \frac{d(d+1)}{2}\, h_K^2 |\xi|^2 = \frac{d}{2(d+2)}h_K^2 |\xi|^2,
\end{align*}
which completes the proof.
\end{proof}


\subsection{\texorpdfstring{\boldmath $L^p$}{Lp} and \texorpdfstring{\boldmath $W^{1,p}$}{W1p} convergence}

Recall from \eqref{equ: perturbation equation} that $-\Delta u^* = f$ with $f = \lambda^* u^* - V u^* - \beta (u^*)^3$. Analogously, by \eqref{equ: mass-lumped discrete GPEVP}, the discrete ground state $u_h^{*,\ell}$ satisfies
\begin{align*}
    (\nabla u_h^{*,\ell}, \nabla v_h)_{L^2(\Omega)} = (f_h^{\ell}, v_h)_{\ell^2} \qquad \forall\, v_h \in V_h^0,
\end{align*}
where $f_h^{\ell} := \lambda_h^{*,\ell}u_h^{*,\ell} - Vu_h^{*,\ell} - \beta (u_h^{*,\ell})^3$.
Let $\delta_h := I_h f_h^{\ell} - I_h f \in V_h$.
\begin{lemma}\label{lem: delta_h}
    There exists $h_0 > 0$ such that for all $0 < h \le h_0$,
    \begin{align*}
        \|\delta_h\|_{L^2(\Omega)} \le C h^2.
    \end{align*}
\end{lemma}
\begin{proof}
By Lemma~\ref{lem: norm equivalence}, we have $\|\delta_h\|_{L^2(\Omega)} \le \|\delta_h\|_{\ell^2}$. Since $\delta_h$ and $f_h^{\ell} - f$ coincide at all nodes,
\begin{align*}
    \|\delta_h\|_{\ell^2} = \|f_h^{\ell} - f\|_{\ell^2}.
\end{align*}
Here, $\|\cdot\|_{\ell^2}$ is extended to $C(\overline{\Omega})$ by the same formula, and the triangle inequality still holds. Using $\|u_h^{*,\ell}\|_{\ell^2} = 1$ and Proposition~\ref{prop: mass-lumped FEM convergence}, we obtain
\begin{align*}
    \|f_h^{\ell} - f\|_{\ell^2} & \le \|\lambda^*(u^* - u_h^{*,\ell})\|_{\ell^2} + |\lambda^* - \lambda_h^{*,\ell}|\,\|u_h^{*,\ell}\|_{\ell^2} + \|V(u^* - u_h^{*,\ell})\|_{\ell^2} \\
    & \quad + \beta \big\|\big((u^*)^2 + u^*u_h^{*,\ell} + (u_h^{*,\ell})^2\big)(u^* - u_h^{*,\ell})\big\|_{\ell^2} \\
    & \le C \|u^* - u_h^{*,\ell}\|_{\ell^2} + Ch^2 + \|V\|_{L^{\infty}(\Omega)}\|u^* - u_h^{*,\ell}\|_{\ell^2} \\
    & \quad + C \big\|(u^*)^2 + u^*u_h^{*,\ell} + (u_h^{*,\ell})^2\big\|_{L^{\infty}(\Omega)} \|u^* - u_h^{*,\ell}\|_{\ell^2} \\
    & \le C \|u^* - u_h^{*,\ell}\|_{\ell^2} + Ch^2.
\end{align*}
Since $u^*$ and $I_hu^*$ coincide at all nodes, Lemma~\ref{lem: norm equivalence}, the interpolation estimate, and Proposition~\ref{prop: mass-lumped FEM convergence} yield
\begin{align*}
    \|u^* - u_h^{*,\ell}\|_{\ell^2} = \|I_hu^* - u_h^{*,\ell}\|_{\ell^2} & \le C \|I_hu^* - u_h^{*,\ell}\|_{L^2(\Omega)} \\
    &\le C \big(\|u^* - I_hu^*\|_{L^2(\Omega)} + \|u^* - u_h^{*,\ell}\|_{L^2(\Omega)}\big) \le Ch^2.
\end{align*}
Combining the above estimates completes the proof.
\end{proof}

Let $\theta_h^{\ell} := u_h^{*,\ell} - R_hu^* \in V_h^0$. By the definition of $R_h$ and \eqref{equ: perturbation equation}, we have for all $v_h \in V_h^0$,
\begin{align*}
    (\nabla \theta_h^{\ell}, \nabla v_h)_{L^2(\Omega)} & = (f_h^{\ell},v_h)_{\ell^2} - (f,v_h)_{L^2(\Omega)} \\
    & = (\delta_h, v_h)_{\ell^2} - (f - I_hf,v_h)_{L^2(\Omega)} + \big[(I_hf, v_h)_{\ell^2} - (I_hf, v_h)_{L^2(\Omega)}\big].
\end{align*}
For the first term, by the Riesz representation theorem on $V_h^0$, there exists a unique $q_h \in V_h^0$ such that
\begin{align*}
    (q_h,v_h)_{L^2(\Omega)} = (\delta_h, v_h)_{\ell^2} \qquad \forall\, v_h \in V_h^0.
\end{align*}
Taking $v_h = q_h$ and using the Cauchy--Schwarz inequality together with Lemma~\ref{lem: norm equivalence}, we obtain
\begin{align*}
    \|q_h\|_{L^2(\Omega)}^2 \le \|\delta_h\|_{\ell^2}\|q_h\|_{\ell^2} \le (d+2)\|\delta_h\|_{L^2(\Omega)} \|q_h\|_{L^2(\Omega)}.
\end{align*}
Hence, Lemma~\ref{lem: delta_h} implies $\|q_h\|_{L^2(\Omega)} \le Ch^2$.
For the third term, applying Lemma~\ref{lem: consistency} on each element $K \in \mathcal{T}_h$ yields
\begin{align*}
    (I_hf, v_h)_{\ell^2} - (I_hf, v_h)_{L^2(\Omega)} = \sum_{K \in \mathcal{T}_h} \int_{K} (D_K\nabla I_hf)\cdot\nabla v_h \,\mathrm{d}x = (\bm{F}_h, \nabla v_h)_{L^2(\Omega)},
\end{align*}
where $\bm{F}_h$ is defined elementwise by $\bm{F}_h|_K := D_K\nabla I_hf|_K$ for $K \in \mathcal{T}_h$. Since $D_K$ is symmetric positive semidefinite, Lemma~\ref{lem: consistency} yields
\begin{align*}
    \|\bm{F}_h\|_{L^{\infty}(\Omega)} = \max_{K \in \mathcal{T}_h} \|D_K\nabla I_hf\|_{L^{\infty}(K)} \le Ch^{2} \|\nabla I_hf\|_{L^{\infty}(\Omega)}.
\end{align*}
Since $V \in W^{2,\infty}(\Omega)$ and $u^* \in W^{2,\infty}(\Omega)$ by Proposition~\ref{prop: W2infty regularity}, we have $f \in W^{2,\infty}(\Omega)$. The interpolation estimate then gives
\begin{align*}
    \|\nabla I_hf\|_{L^{\infty}(\Omega)} &\le \|\nabla f\|_{L^{\infty}(\Omega)} + \|\nabla(f - I_hf)\|_{L^{\infty}(\Omega)}\\
    &\le \|\nabla f\|_{L^{\infty}(\Omega)} + Ch\|f\|_{W^{2,\infty}(\Omega)} \le C,
\end{align*}
and hence $\|\bm{F}_h\|_{L^{\infty}(\Omega)} \le Ch^2$.
Combining the above estimates with the interpolation estimate, Lemma~\ref{lem: general framework}, and the inverse inequality, we obtain the following result.
\begin{lemma}\label{lem: superconvergence mass-lumped}
    There exists $h_0 > 0$ such that for all $0 < h \le h_0$,
    \begin{align*}
        \|\nabla \theta_h^{\ell}\|_{L^6(\Omega)} + \|\theta_h^{\ell}\|_{L^{\infty}(\Omega)} \le Ch^2, \qquad \|\nabla\theta_h^{\ell}\|_{L^{\infty}(\Omega)} \le Ch^{2 - \frac{d}{6}}.
    \end{align*}
\end{lemma}

Proceeding as in the proofs of Theorems~\ref{thm: p convergence} and~\ref{thm: Linfty and W1infty convergence}, with Lemma~\ref{lem: superconvergence mass-lumped} in place of Lemmas~\ref{lem: superconvergence le 6} and~\ref{lem: Linfty estimates of theta_h}, we obtain the following result.
\begin{theorem}\label{thm: L^p convergence of mass-lumped}
    There exists $h_0 > 0$ such that for all $0 < h \le h_0$,
    \begin{align*}
        \|u^* - u_h^{*,\ell}\|_{L^p(\Omega)} + h \|\nabla(u^* - u_h^{*,\ell})\|_{L^p(\Omega)} \le C_ph^2 \qquad \text{for } 1 \le p < \infty,
    \end{align*}
    and
    \begin{align*}
        \|u^* - u_h^{*,\ell}\|_{L^{\infty}(\Omega)} \le Ch^2|\ln h|, \qquad \|\nabla (u^* - u_h^{*,\ell})\|_{L^{\infty}(\Omega)} \le C h.
    \end{align*}
\end{theorem}

\subsection{Superconvergence}

With the help of Theorem~\ref{thm: L^p convergence of mass-lumped}, the estimate of $\delta_h$ in Lemma~\ref{lem: delta_h} can be improved as follows.
\begin{lemma}\label{lem: delta_h Lq}
    For every $2 \le q < \infty$, there exists $h_0 > 0$ such that for all $0 < h \le h_0$,
    \begin{align*}
        \|\delta_h\|_{L^q(\Omega)} \le C_q h^2.
    \end{align*}
\end{lemma}

\begin{proof}
    As in the proof of Lemma~\ref{lem: delta_h}, using Lemma~\ref{lem: norm equivalence} and the fact that $u^*$ and $I_hu^*$ coincide at all nodes, we obtain
    \begin{align*}
        \|\delta_h\|_{L^q(\Omega)} \le C_q\|\delta_h\|_{\ell^q} & \le C_q|\lambda_h^{*,\ell} - \lambda^*|\,\|u_h^{*,\ell}\|_{\ell^q} + C_q \|u_h^{*,\ell} - I_hu^*\|_{\ell^q} \\
        & \le C_q h^2 \|u_h^{*,\ell}\|_{L^q(\Omega)} + C_q\|u_h^{*,\ell} - I_hu^*\|_{L^q(\Omega)} \le C_qh^2,
    \end{align*}
    where the last step follows from Proposition~\ref{prop: mass-lumped FEM convergence}, the interpolation estimate, and Theorem~\ref{thm: L^p convergence of mass-lumped}.
    \end{proof}

    To exploit Lemma~\ref{lem: delta_h Lq}, we rewrite the equation for $\theta_h^{\ell}$ as follows: for all $v_h \in V_h^0$,
    \begin{align*}
        (\nabla \theta_h^{\ell}, \nabla v_h)_{L^2(\Omega)} & = (f_h^{\ell},v_h)_{\ell^2} - (f,v_h)_{L^2(\Omega)} \\
        & = (\delta_h, v_h)_{L^2(\Omega)} + \big[(\delta_h, v_h)_{\ell^2} - (\delta_h, v_h)_{L^2(\Omega)}\big] \\
        & \quad - (f - I_hf,v_h)_{L^2(\Omega)} + \big[(I_hf, v_h)_{\ell^2} - (I_hf, v_h)_{L^2(\Omega)}\big].
    \end{align*}
    For the second term, Lemma~\ref{lem: consistency} yields
    \begin{align*}
        (\delta_h, v_h)_{\ell^2} - (\delta_h, v_h)_{L^2(\Omega)} = \sum_{K \in \mathcal{T}_h}\int_{K} (D_K\nabla \delta_h)\cdot\nabla v_h \,\mathrm{d}x = (\bm{G}_h, \nabla v_h)_{L^2(\Omega)},
    \end{align*}
    where $\bm{G}_h|_K := D_K\nabla \delta_h|_K$ for $K \in \mathcal{T}_h$. Moreover, $\bm{G}_h$ is of higher order.
    \begin{lemma}\label{lem: high order}
        For every $2 \le q < \infty$, there exists $h_0 > 0$ such that for all $0 < h \le h_0$,
        \begin{align*}
            \|\bm{G}_h\|_{L^q(\Omega)} \le C_qh^3.
        \end{align*}
    \end{lemma}
    \begin{proof}
        Since $D_K$ is symmetric, Lemma~\ref{lem: consistency} together with the inverse inequality and Lemma~\ref{lem: delta_h Lq} yields
        \begin{align*}
            \|\bm{G}_h\|_{L^q(\Omega)} = \Big(\sum_{K \in \mathcal{T}_h} \int_K |D_K \nabla \delta_h|^q \,\mathrm{d}x \Big)^{\frac1q}
            \le C h^2 \|\nabla \delta_h\|_{L^q(\Omega)} \le C_qh^3,
        \end{align*}
        which completes the proof.
    \end{proof}

    Let $g_h := \delta_h + (I_hf - f)$ and $\tilde{\bm{F}}_h := \bm{F}_h + \bm{G}_h$. By Lemmas~\ref{lem: delta_h Lq} and~\ref{lem: high order} and the interpolation estimate, for every $2 \le q < \infty$,
    \begin{align*}
        \|g_h\|_{L^q(\Omega)} + \|\tilde{\bm{F}}_h\|_{L^q(\Omega)} \le C_qh^2.
    \end{align*}
    We are now in a position to state the superconvergence result for the mass-lumped FEM.
    \begin{theorem}\label{thm: superconvergence of mass-lumped}
        For every $1 \le q < \infty$, there exists $h_0 > 0$ such that for all $0 < h \le h_0$,
        \begin{align*}
            \|\nabla (u_h^{*,\ell} - R_hu^*)\|_{L^q(\Omega)} \le C_qh^2.
        \end{align*}
    \end{theorem}
    \begin{proof}
        It suffices to consider $2 \le q < \infty$, since the case $1 \le q < 2$ then follows from H\"older's inequality. By the above decomposition, $\theta_h^{\ell}$ satisfies
        \begin{align*}
            (\nabla \theta_h^{\ell}, \nabla v_h)_{L^2(\Omega)} = (g_h,v_h)_{L^2(\Omega)} + (\tilde{\bm{F}}_h,\nabla v_h)_{L^2(\Omega)} \qquad \forall\, v_h \in V_h^0.
        \end{align*}
        Let $w_{g_h}, w_{\tilde{\bm{F}}_h} \in H_0^1(\Omega)$ satisfy
        \begin{align*}
            (\nabla w_{g_h}, \nabla v)_{L^2(\Omega)} &= (g_h,v)_{L^2(\Omega)} && \forall\, v \in H_0^1(\Omega), \\
            (\nabla w_{\tilde{\bm{F}}_h}, \nabla v)_{L^2(\Omega)} &= (\tilde{\bm{F}}_h,\nabla v)_{L^2(\Omega)} && \forall\, v \in H_0^1(\Omega),
        \end{align*}
        so that $\theta_h^{\ell} = R_h(w_{g_h} + w_{\tilde{\bm{F}}_h})$. By elliptic regularity and Lemma~\ref{lem: regularity}, respectively, we have
        \begin{align*}
            \|w_{g_h}\|_{W^{2,4}(\Omega)} \le C\|g_h\|_{L^4(\Omega)} \le C h^2, \qquad \|w_{\tilde{\bm{F}}_h}\|_{W^{1,q}(\Omega)} \le C_q\|\tilde{\bm{F}}_h\|_{L^q(\Omega)} \le C_qh^2.
        \end{align*}
        Hence, Proposition~\ref{prop: stability of the Ritz projection} and the Sobolev embedding $W^{2,4}(\Omega) \hookrightarrow W^{1,q}(\Omega)$ for $d \le 3$ yield
        \begin{align*}
            \|\nabla \theta_h^{\ell}\|_{L^q(\Omega)} & \le C \|\nabla (w_{g_h} + w_{\tilde{\bm{F}}_h})\|_{L^q(\Omega)} \\
            &\le C_q \|w_{g_h}\|_{W^{2,4}(\Omega)} + C \|w_{\tilde{\bm{F}}_h}\|_{W^{1,q}(\Omega)} \le C_qh^2,
        \end{align*}
        which completes the proof.
    \end{proof}

\begin{remark}\label{Remark:1D}
    For $d = 1$, the superconvergence also holds for $q = \infty$. Indeed, $w_{\tilde{\bm{F}}_h}$ can be decomposed as $w_{\tilde{\bm{F}}_h} = w_{\bm{F}_h} + w_{\bm{G}_h}$, where $w_{\bm{F}_h}$ and $w_{\bm{G}_h}$ are defined analogously with $\bm{F}_h$ and $\bm{G}_h$, respectively. In one dimension, these problems can be solved explicitly, which yields $\|w'_{\bm{F}_h}\|_{L^{\infty}(\Omega)} \le Ch^2$ and $\|w'_{\bm{G}_h}\|_{L^{q}(\Omega)} \le C_qh^3$. Hence, $\theta_h^{\ell} = R_h(w_{g_h} + w_{\bm{F}_h}) + R_hw_{\bm{G}_h}$ with $\|R_h(w_{g_h} + w_{\bm{F}_h})\|_{W^{1,\infty}(\Omega)} \le Ch^2$ and $\|R_hw_{\bm{G}_h}\|_{W^{1,q}(\Omega)} \le C_qh^3$, and the inverse inequality yields $\|R_hw_{\bm{G}_h}\|_{W^{1,\infty}(\Omega)} \le C_qh^{3-\frac1q} \le C_qh^2$.
\end{remark}

\begin{remark}\label{Remark:2 3D}
    For $d \in \{2,3\}$, the superconvergence also holds for $q = \infty$ if the mesh satisfies $D_K = D$ for all $K \in \mathcal{T}_h$ and some matrix $D$ independent of $K$. In this case, $\bm{F}_h$ can be decomposed as
    \begin{align*}
        \bm{F}_h = D \nabla I_h f = D \nabla f + D\nabla (I_hf - f) =: \bm{F}_h^1 + \bm{F}_h^2.
    \end{align*}
    For the first term, integration by parts yields
    \begin{align*}
        (\bm{F}_h^1,\nabla v)_{L^2(\Omega)} = \big(- \operatorname{div}(D\nabla f),v\big)_{L^2(\Omega)} \qquad \forall\, v \in H_0^1(\Omega),
    \end{align*}
    where $\|\operatorname{div}(D\nabla f)\|_{L^\infty(\Omega)} \le Ch^2$, so that this term can be treated in the same way as $g_h$. Since $\|\bm{F}_h^2\|_{L^q(\Omega)} + \|\bm{G}_h\|_{L^q(\Omega)} \le C_q h^3$, the terms $\bm{F}_h^2$ and $\bm{G}_h$ can be handled as in Remark~\ref{Remark:1D}, which yields the superconvergence for $q = \infty$.
\end{remark}





\section{Numerical experiments}\label{sec: numerical experiments}

In this section, we present two numerical examples in two dimensions to compare the standard FEM and the mass-lumped FEM. The errors are measured in the $L^p$- and $W^{1,p}$-norms for $p = 1, 2, 4, 8, \infty$. Both methods use the continuous piecewise linear finite element space $V_h^0$ on a triangulation obtained by splitting each square of a uniform square grid along its southwest-to-northeast diagonal. Note that this triangulation satisfies $D_K = D$ for all $K \in \mathcal{T}_h$. The discrete nonlinear eigenvalue problems are solved by a Newton method for the nodal values and the eigenvalue simultaneously, combined with a backtracking line search. On each refined mesh, the initial iterate is obtained by interpolating the converged solution on the preceding mesh and is then normalized with respect to the $L^2$ inner product for the standard FEM and with respect to the discrete inner product $(\cdot,\cdot)_{\ell^2}$ for the mass-lumped FEM. The iteration is stopped once the relative residual is below $2 \times 10^{-11}$ and the normalization constraint is satisfied up to $10^{-12}$.
\subsection{Smooth ground state constructed from two sine modes}

Let $\Omega = (0,1)^2$ and $\beta = 10$. We choose the potential $V$ such that the exact ground state and the associated eigenvalue are given by
\begin{align*}
    u^*(x,y) = \frac{8}{\sqrt{17}} \Big(\sin\pi x + \frac14\sin3\pi x\Big)\sin\pi y, \qquad \lambda^* = 120,
\end{align*}
where the normalization factor ensures $\|u^*\|_{L^2(\Omega)} = 1$. More precisely, we set
\begin{align*}
    V(x,y) = 120 - q(x) - 10\,u^*(x,y)^2, \qquad q(x) = 2\pi^2\frac{9+10\cos2\pi x}{5+2\cos2\pi x}.
\end{align*}
A direct computation shows that $-\Delta u^* = q\,u^*$, and hence $(\lambda^*, u^*)$ solves \eqref{equ: GPEVP}. Moreover, $V$ is smooth on $\overline{\Omega}$ and satisfies
\begin{align*}
    V(x,y) \ge 120 - \frac{38\pi^2}{7} - \frac{1000}{17} > 7.59.
\end{align*}
Since $u^*$ is strictly positive in $\Omega$ and the positive normalized eigenfunction is unique, $u^*$ is the ground state; see, e.g., \cite{cances2010numerical}. The mesh size is $h = \sqrt{2}/N$, where $N$ denotes the number of subdivisions in each coordinate direction. On the coarsest mesh, Newton's method is initialized with the nodal interpolant $I_hu^*$, normalized with respect to the corresponding inner product.

Figure~\ref{fig:convergence-A} shows the errors in the $L^p$- and $W^{1,p}$-norms for $p = 1, 2, 4, 8, \infty$. For both methods, the errors in the $L^p$-norm and the $W^{1,p}$-norm converge with second and first order, respectively, which is consistent with Theorems~\ref{thm: p convergence}, \ref{thm: Linfty and W1infty convergence}, and~\ref{thm: L^p convergence of mass-lumped}. Moreover, Table~\ref{tab:ritz-errors} reports the errors $\|u_h - R_hu^*\|_{W^{1,p}(\Omega)}$, where $u_h$ denotes $u_h^*$ or $u_h^{*,\ell}$. For both methods and all tested values of $p$, these errors converge with approximately second order, i.e., one order higher than $\|u^* - u_h\|_{W^{1,p}(\Omega)}$. This confirms the superconvergence results in Theorems~\ref{thm: superconvergence of standard FEM} and~\ref{thm: superconvergence of mass-lumped}. For the mass-lumped FEM with $p = \infty$, the observed superconvergence agrees with Remark~\ref{Remark:2 3D}.

\begin{figure}[tbp]
    \centering
    \includegraphics[width=0.48\linewidth]{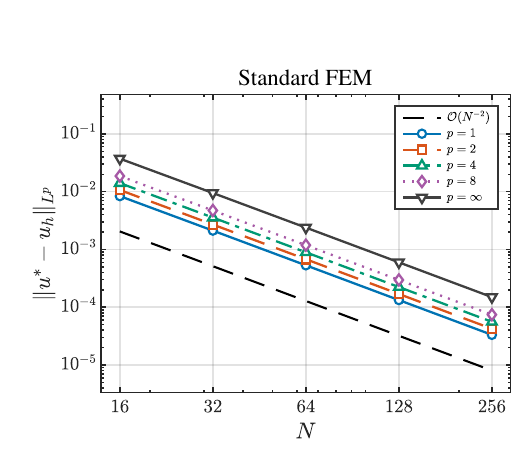}
    \hfill
    \includegraphics[width=0.48\linewidth]{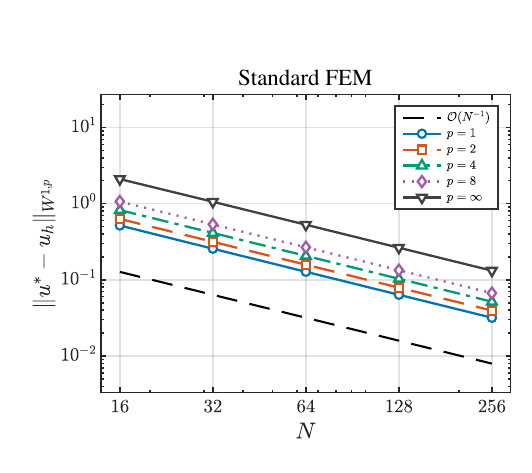}

    \par\medskip

    \includegraphics[width=0.48\linewidth]{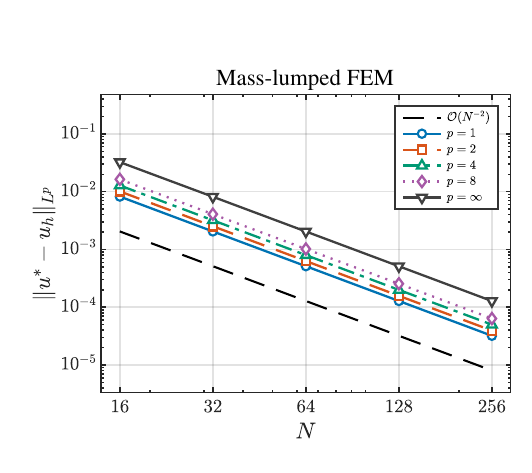}
    \hfill
    \includegraphics[width=0.48\linewidth]{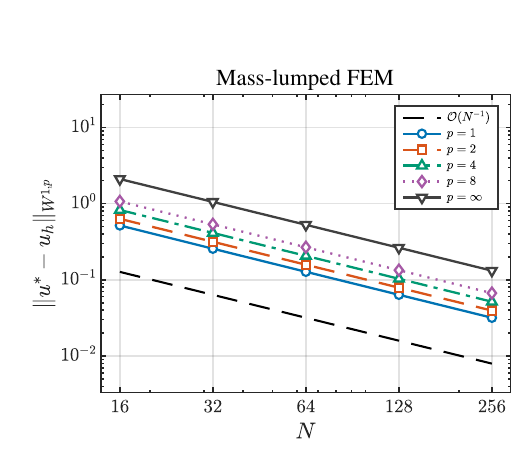}

    \caption{Errors of the standard FEM (top) and the mass-lumped FEM (bottom) in the $L^p$-norm (left) and the $W^{1,p}$-norm (right) for $p = 1,2,4,8,\infty$. The dashed lines indicate the reference rates $\mathcal{O}(N^{-2})$ (left) and $\mathcal{O}(N^{-1})$ (right).}
    \label{fig:convergence-A}
\end{figure}

\begin{table}[htbp]
\centering
\small
\setlength{\tabcolsep}{9pt}
\renewcommand{\arraystretch}{1.1}
\caption{Errors $\|u_h - R_hu^*\|_{W^{1,p}(\Omega)}$ and convergence rates, where $u_h$ denotes $u_h^*$ for the standard FEM and $u_h^{*,\ell}$ for the mass-lumped FEM.}
\label{tab:ritz-errors}
\begin{tabular}{@{}rr cc cc@{}}
\toprule
& & \multicolumn{2}{c}{Standard FEM} & \multicolumn{2}{c}{Mass-lumped FEM} \\
\cmidrule(lr){3-4}\cmidrule(lr){5-6}
$p$ & $N$ & $\|u_h^*-R_hu^*\|_{W^{1,p}}$ & Rate & $\|u_h^{*,\ell}-R_hu^*\|_{W^{1,p}}$ & Rate \\
\midrule
\multirow{5}{*}{$1$} & $16$ & $5.9133\times10^{-2}$ & -- & $6.9708\times10^{-2}$ & -- \\
 & $32$ & $1.4961\times10^{-2}$ & 1.98 & $1.7588\times10^{-2}$ & 1.99 \\
 & $64$ & $3.7519\times10^{-3}$ & 2.00 & $4.4041\times10^{-3}$ & 2.00 \\
 & $128$ & $9.3875\times10^{-4}$ & 2.00 & $1.1015\times10^{-3}$ & 2.00 \\
 & $256$ & $2.3473\times10^{-4}$ & 2.00 & $2.7541\times10^{-4}$ & 2.00 \\
\midrule
\multirow{5}{*}{$2$} & $16$ & $5.5839\times10^{-2}$ & -- & $7.2273\times10^{-2}$ & -- \\
 & $32$ & $1.4111\times10^{-2}$ & 1.98 & $1.8142\times10^{-2}$ & 1.99 \\
 & $64$ & $3.5377\times10^{-3}$ & 2.00 & $4.5403\times10^{-3}$ & 2.00 \\
 & $128$ & $8.8506\times10^{-4}$ & 2.00 & $1.1354\times10^{-3}$ & 2.00 \\
 & $256$ & $2.2130\times10^{-4}$ & 2.00 & $2.8386\times10^{-4}$ & 2.00 \\
\midrule
\multirow{5}{*}{$4$} & $16$ & $6.1020\times10^{-2}$ & -- & $8.4667\times10^{-2}$ & -- \\
 & $32$ & $1.5423\times10^{-2}$ & 1.98 & $2.1239\times10^{-2}$ & 2.00 \\
 & $64$ & $3.8669\times10^{-3}$ & 2.00 & $5.3145\times10^{-3}$ & 2.00 \\
 & $128$ & $9.6742\times10^{-4}$ & 2.00 & $1.3289\times10^{-3}$ & 2.00 \\
 & $256$ & $2.4190\times10^{-4}$ & 2.00 & $3.3225\times10^{-4}$ & 2.00 \\
\midrule
\multirow{5}{*}{$8$} & $16$ & $6.8112\times10^{-2}$ & -- & $1.0240\times10^{-1}$ & -- \\
 & $32$ & $1.7240\times10^{-2}$ & 1.98 & $2.5721\times10^{-2}$ & 1.99 \\
 & $64$ & $4.3243\times10^{-3}$ & 2.00 & $6.4363\times10^{-3}$ & 2.00 \\
 & $128$ & $1.0820\times10^{-3}$ & 2.00 & $1.6095\times10^{-3}$ & 2.00 \\
 & $256$ & $2.7056\times10^{-4}$ & 2.00 & $4.0239\times10^{-4}$ & 2.00 \\
\midrule
\multirow{5}{*}{$\infty$} & $16$ & $8.9358\times10^{-2}$ & -- & $1.5084\times10^{-1}$ & -- \\
 & $32$ & $2.2908\times10^{-2}$ & 1.96 & $3.9699\times10^{-2}$ & 1.93 \\
 & $64$ & $5.8751\times10^{-3}$ & 1.96 & $1.0098\times10^{-2}$ & 1.98 \\
 & $128$ & $1.4887\times10^{-3}$ & 1.98 & $2.5405\times10^{-3}$ & 1.99 \\
 & $256$ & $3.7476\times10^{-4}$ & 1.99 & $6.3680\times10^{-4}$ & 2.00 \\
\bottomrule
\end{tabular}

\medskip
\end{table}

\subsection{Ground state in a harmonic potential}

This example concerns the ground state of a repulsively interacting Bose--Einstein condensate confined by an isotropic harmonic trap. Let $\Omega = (-6,6)^2$, $\beta = 100$, and
\begin{align*}
    V(x,y) = \frac{x^2+y^2}{2}.
\end{align*}
The mesh size is $h = 12\sqrt{2}/N$, where $N$ denotes the number of subdivisions in each coordinate direction. On the coarsest mesh, the initial iterate is the finite element function whose interior nodal values coincide with those of
\begin{align*}
    u_0(x,y) = \exp\Big(-\frac{x^2+y^2}{8}\Big)
\end{align*}
and whose boundary nodal values vanish, normalized with respect to the corresponding inner product. Since the exact ground state is not available, the error on the mesh of size $h$ is estimated by the difference $u_h - u_{h/2}$ between the numerical solutions on two successive meshes. Since the meshes are nested, $u_h$ can be represented exactly on the finer mesh, and all norms are evaluated there.

Figure~\ref{fig:convergence-B} shows the results for both methods. For $p = 1, 2, 4, 8, \infty$, the differences $\|u_h - u_{h/2}\|_{L^p(\Omega)}$ and $\|u_h - u_{h/2}\|_{W^{1,p}(\Omega)}$ converge with approximately second and first order, respectively. These results further confirm the convergence rates established in Theorems~\ref{thm: p convergence}, \ref{thm: Linfty and W1infty convergence}, and~\ref{thm: L^p convergence of mass-lumped}.
\begin{figure}[tbp]
    \centering
    \includegraphics[width=0.48\linewidth]{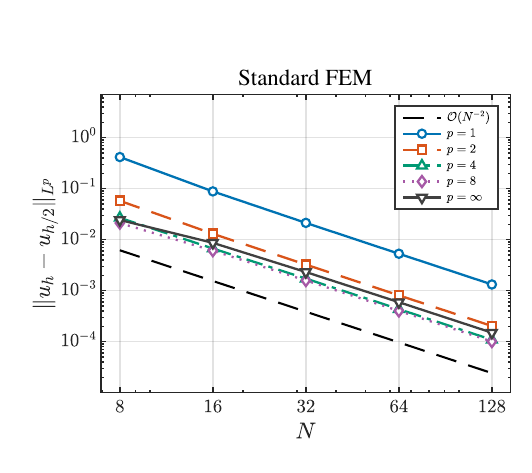}
    \hfill
    \includegraphics[width=0.48\linewidth]{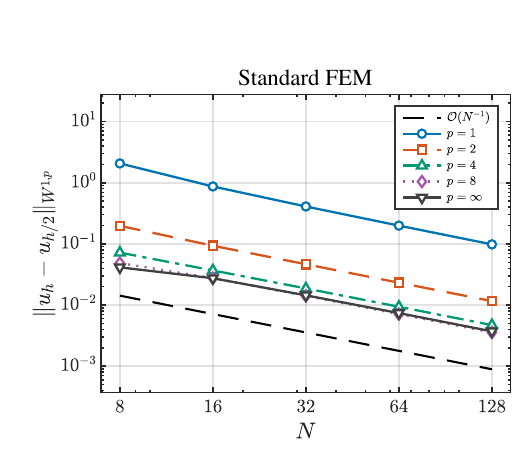}

    \par\medskip

    \includegraphics[width=0.48\linewidth]{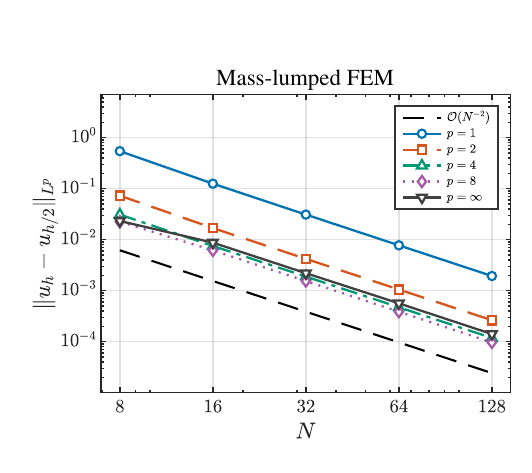}
    \hfill
    \includegraphics[width=0.48\linewidth]{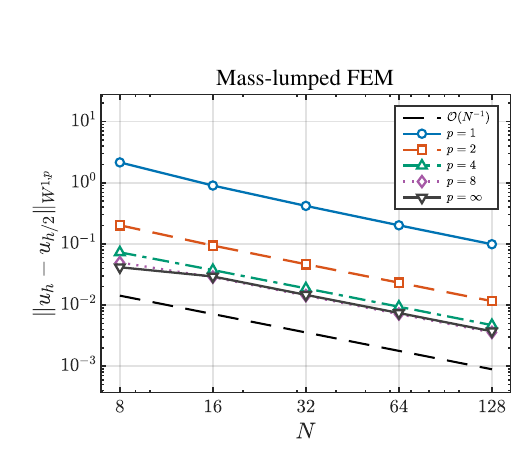}

    \caption{Differences $u_h - u_{h/2}$ between the numerical solutions on successive meshes for the standard FEM (top) and the mass-lumped FEM (bottom) in the $L^p$-norm (left) and the $W^{1,p}$-norm (right) for $p = 1,2,4,8,\infty$. The dashed lines indicate the reference rates $\mathcal{O}(N^{-2})$ (left) and $\mathcal{O}(N^{-1})$ (right).}
    \label{fig:convergence-B}
\end{figure}

\section{Concluding remarks}\label{sec: conclusion}

In this paper, we have analyzed the standard and the mass-lumped linear finite element methods for the Gross--Pitaevskii eigenvalue problem. For the standard FEM, we have proved that the discrete ground state is unique up to sign for sufficiently small mesh sizes. For both methods, we have derived optimal convergence rates in the $L^p$-norm for $1 \le p < \infty$ and in the $W^{1,p}$-norm for $1 \le p \le \infty$, as well as an almost optimal convergence rate, up to a logarithmic factor, in the $L^\infty$-norm. Furthermore, we have shown that the discrete ground states superconverge to the Ritz projection of the exact ground state in the $W^{1,p}$-seminorm. For the standard FEM, this holds for $1 \le p \le \infty$, whereas for the mass-lumped FEM, it holds for $1 \le p < \infty$ in general and for $p = \infty$ if $d = 1$ or if $D_K$ is independent of $K$. The latter case is confirmed by our numerical experiments. Whether the superconvergence of the mass-lumped FEM holds for $p = \infty$ on general meshes in two and three dimensions remains an open question.

\section*{Declaration of generative AI use}
The research problem considered in this paper, the main ideas and proof strategies underlying the error analysis, and the overall structure of the manuscript were conceived and developed by the authors. Generative AI tools, ChatGPT (OpenAI) and Claude (Anthropic), were used to assist with the language editing of the manuscript, the verification of mathematical derivations, the design of numerical examples, and the review of related results in the literature. All AI-generated suggestions were critically examined and, where necessary, revised by the authors. The authors have verified all mathematical arguments, references, and numerical results, and take full responsibility for the content of this paper.
\bibliographystyle{siamplain}
\bibliography{references}

\end{document}